\documentclass[letterpaper, 11pt,  reqno]{amsart}

\usepackage{amsmath,amssymb,amscd,amsthm,amsxtra, esint}
\usepackage{tikz} 

\usepackage{comment}
\usepackage{mathtools}
\usepackage{color}
\usepackage[implicit=true]{hyperref}
\usepackage{scalerel} 

\usepackage{cases}

\usepackage[left=32mm, right=32mm, 
bottom=27mm]{geometry}

\allowdisplaybreaks[2]

\usepackage{color}

\definecolor{gr}{rgb}   {0.,   0.69,   0.23 }
\definecolor{bl}{rgb}   {0.,   0.5,   1. }
\definecolor{mg}{rgb}   {0.85,  0.,    0.85}
\definecolor{yl}{rgb}   {0.8,  0.7,   0.}
\definecolor{or}{rgb}  {0.7,0.2,0.2}

\usetikzlibrary{shapes.misc}
\usetikzlibrary{shapes.symbols}
\usetikzlibrary{shapes.geometric}

\tikzset{
	ddot/.style={circle,fill=white,draw=black,inner sep=0pt,minimum size=0.8mm},
	>=stealth,
	}

\tikzset{
	ddot2/.style={circle,fill=black,draw=black,inner sep=0pt,minimum size=0.8mm},
	>=stealth,
	}

\newtheorem{theorem}{Theorem} [section]

\newtheorem{lemma}[theorem]{Lemma}
\newtheorem{proposition}[theorem]{Proposition}
\newtheorem{remark}[theorem]{Remark}

\newtheorem{corollary}[theorem]{Corollary}

\newcommand{\I}{\mathcal{I}}

\newcommand{\noi}{\noindent}
\newcommand{\Z}{\mathbb{Z}}
\newcommand{\R}{\mathbb{R}}

\newcommand{\T}{\mathbb{T}}

\let\Re=\undefined\DeclareMathOperator*{\Re}{Re}
\let\Im=\undefined\DeclareMathOperator*{\Im}{Im}

\let\P= \undefined
\newcommand{\P}{\mathbf{P}}

\newcommand{\F}{\mathcal{F}}

\newcommand{\al}{\alpha}

\newcommand{\dl}{\delta}

\newcommand{\nb}{\nabla}

\newcommand{\Dl}{\Delta}
\newcommand{\eps}{\varepsilon}

\newcommand{\g}{\gamma}
\newcommand{\G}{\Gamma}
\newcommand{\ld}{\lambda}

\newcommand{\s}{\sigma}

\newcommand{\ft}{\widehat}

\newcommand{\wt}{\widetilde}
\newcommand{\cj}{\overline}

\newcommand{\dt}{\partial_t}

\newcommand{\ta}{\theta}

\renewcommand{\l}{\ell}
\renewcommand{\o}{\omega}
\renewcommand{\O}{\Omega}

\newcommand{\les}{\lesssim}

\newcommand{\jb}[1]
{\langle #1 \rangle}

\newcommand{\jbb}[1]
{[\hspace{-0.6mm}[ #1 ]\hspace{-0.6mm}]}

\renewcommand{\b}{\beta}
\newcommand{\ind}{\mathbf 1}

\newcommand{\N}{\mathbb{N}}

\renewcommand{\H}{\mathcal{H}}

\newtheorem*{ackno}{Acknowledgements}

\numberwithin{equation}{section}
\numberwithin{theorem}{section}

\newcommand{\pe}{\mathbin{\scaleobj{0.7}{\tikz \draw (0,0) node[shape=circle,draw,inner sep=0pt,minimum size=8.5pt] {\scriptsize  $=$};}}}
\newcommand{\pl}{\mathbin{\scaleobj{0.7}{\tikz \draw (0,0) node[shape=circle,draw,inner sep=0pt,minimum size=8.5pt] {\scriptsize $<$};}}}
\newcommand{\pg}{\mathbin{\scaleobj{0.7}{\tikz \draw (0,0) node[shape=circle,draw,inner sep=0pt,minimum size=8.5pt] {\scriptsize $>$};}}}

\makeatletter
\@namedef{subjclassname@2020}{%
  \textup{2020} Mathematics Subject Classification}
\makeatother

\begin{document}
\baselineskip = 14pt

\title[GWP of 2-$d$ stochastic viscous NLW]{Global well-posedness of the two-dimensional stochastic viscous nonlinear wave equations}

\author[R.~Liu]{Ruoyuan Liu}

\address{
Ruoyuan Liu,  School of Mathematics\\
The University of Edinburgh\\
and The Maxwell Institute for the Mathematical Sciences\\
James Clerk Maxwell Building\\
The King's Buildings\\
Peter Guthrie Tait Road\\
Edinburgh\\ 
EH9 3FD\\
United Kingdom}

\email{ruoyuan.liu@ed.ac.uk}

\begin{abstract}
We study well-posedness of viscous nonlinear wave equations (vNLW) on the two-dimensional torus with  
a 
stochastic forcing. In particular, we prove pathwise global well-posedness of the stochastic defocusing vNLW with an additive stochastic forcing $D^\al \xi$, where $\al < \frac 12$ and $\xi$ denotes the space-time white noise.

\end{abstract}

\date{\today}



\maketitle

\tableofcontents

\newpage
\section{Introduction}
\label{SEC:1}

\subsection{Viscous nonlinear wave equations}

In this paper, we consider the following nonlinear wave equation (NLW) on 
the two-dimensional torus
$\T^2 = (\R/\Z)^2$, augmented by viscous effects:
\begin{equation}
\begin{cases}
\dt^2 u + (1-\Dl) u + D \dt u + |u|^{p-1}u = D^\al \xi \\
(u, \dt u)|_{t=0} = (u_0, u_1),
\end{cases}
\label{SvNLW}
\end{equation}

\noi
where $p > 1$, $D = |\nb| = \sqrt{-\Dl}$, $\al < \frac 12$, and $\xi$ denotes the (Gaussian) space-time white noise on $\R_+ \times \T^2$. Our main goal in this paper is to prove pathwise global well-posedness of \eqref{SvNLW} in $C(\R_+; H^s (\T^2))$ for some $\al \le \al_p$ and $s \geq s_p$, where $H^s(\T^2)$ is the $L^2$-based Sobolev space on $\T^2$ with regularity $s$ (see Section \ref{SEC:2} for more details).

In \cite{KC1}, Kuan-\v{C}ani\'{c} 
proposed the following viscous NLW on $\R^2$:
\begin{equation}
\dt^2 u - \Dl u + 2\mu D \dt u = F(u),
\label{vNLW_gen}
\end{equation}

\noi
where $\mu > 0$ and $F(u)$ is a general external forcing. This equation typically shows up in fluid-structure interaction problems, such as the interaction between a stretched membrane and a viscous fluid. 
The viscosity term $2\mu D \dt u$ in \eqref{vNLW_gen} comes from the Dirichlet-Neumann operator typically arising in fluid-structure interaction
problems in three dimensions.
See \cite{KC1} and \cite{CKO} for  the derivation of \eqref{vNLW_gen}.
It is easy to see that, when $\mu \ge 1$,  the equation \eqref{vNLW_gen} is purely parabolic (see \cite{CKO, LO}).
On the other hand, when $0 < \mu < 1$,  
the viscous NLW \eqref{vNLW_gen} exhibits an interesting mixture of
dispersive effects and parabolic smoothing.
Since the precise value of $0 < \mu < 1$ does not play an important role, 
we simply set $\mu = \frac 12$. In addition, we consider a defocusing power-type nonlinearity of the form
\[ F(u) = -|u|^{p-1}u, \]
for positive real numbers $p > 1$. This power-type nonlinearity has been studied extensively for nonlinear dispersive equations (see, for example, \cite{TAO}). With $\mu = \frac 12$ and $F(u) = -|u|^{p-1}u$, the general form of vNLW \eqref{vNLW_gen} becomes the following version of vNLW:
\begin{equation}
\dt^2 u - \Dl u + D \dt u + |u|^{p-1}u = 0
\label{vNLW1}
\end{equation}

We now consider the analytical aspects of vNLW \eqref{vNLW1}. Note that as in the case of the usual NLW:
\begin{equation*}
\dt^2 u - \Dl u + |u|^{p-1}u = 0,
\end{equation*}
the viscous NLW \eqref{vNLW1} on $\R^2$ enjoys the following scaling symmetry: $u(t,x) \mapsto u_\ld (t, x) := \ld^{\frac{2}{p-1}} u(\ld t, \ld x)$. Namely, if $u$ is a solution to \eqref{vNLW1}, then $u_\ld$ is also a solution to \eqref{vNLW1} for any $\ld > 0$ with rescaled initial data. This scaling symmetry induces the scaling critical Sobolev regularity $s_{\text{scaling}}$ on $\R^2$ given by 
\[ s_{\text{scaling}} = 1 - \frac{2}{p-1} \]
such that under this scaling symmetry, the homogeneous Sobolev norm on $\R^2$ remains invariant. While there is no scaling symmetry on $\T^2$, the scaling critical regularity $s_{\text{scaling}}$ still plays an important role in studying nonlinear partial differential equations in the periodic setting, especially for dispersive equations. Namely, in both periodic and non-periodic settings, a dispersive equation is usually well-posed in $H^s$ for $s > s_{\text{scaling}}$ and is usually ill-posed in $H^s$ for $s < s_{\text{scaling}}$. On the one hand, there is a good local well-posedness theory for dispersive equations above the scaling regularities (see \cite{BOP2, Oh, Poc17} for the references therein). Moreover, we show in this paper that vNLW \eqref{vNLW1} is locally well-posed in $H^s(\T^2)$ for all $s \geq s_{\text{crit}}$ (with a strict inequality when $p = 3$), where $s_{\text{crit}}$ is defined by
\begin{equation}
s_{\text{crit}} := \max(s_{\text{scaling}}, 0) = \max\bigg( 1-\frac{2}{p-1}, 0 \bigg),
\label{scrit1}
\end{equation}

\noi
for a given $p > 1$. See Appendix \ref{SEC:LWP3}. Here, the second regularity restriction 0 is required to make sense of powers of $u$. On the other hand, many dispersive equations are known to be ill-posed below the scaling critical regularity. Among these ill-posedness results, many of them are in the form of {\it norm inflation} (see \cite{CCT, BT08, CK, Kish19, Oh, OW, CP, Oka, Tz, OOTz, FO}), which is a stronger notion of ill-posedness. In \cite{KC1}, Kuan-\v{C}ani\'{c} proved norm inflation for vNLW \eqref{vNLW1} in $\H^s(\R^d)$ for $0 < s < s_{\text{scaling}}$ and any odd integers $p \geq 3$. Moreover, they pointed out that the viscous term has the potential to slow down the growth of the $H^s$ norm, i.e. to slow down the speed of the norm inflation. For details, see \cite{KC1}. Also, it is of interest to see if norm inflation for vNLW holds in negative Sobolev spaces. See \cite{dO}.

Let us now turn our attention to the viscous NLW with a stochastic forcing.
  In \cite{KC2}, Kuan-\v{C}ani\'{c} studied the stochastic viscous wave equation with a multiplicative noise on $\R^d$, $d = 1, 2$:
\[\dt^2 u - \Dl u + D \dt u = f(u) \xi, \]
where $f$ is Lipschitz and $\xi$ is the (Gaussian) space-time white noise on $\R_+ \times \R^2$.
In \cite{LO}, T. Oh and the author studied (the renormalized version of) 
SvNLW \eqref{SvNLW} with $\al = \frac 12$.
When $\al = \frac 12$, the solution is not a function but is only a distribution
and thus a renormalization on the nonlinearity is required to give a proper meaning to the dynamics
(which in particular forces us to consider $|u|^{p-1}u$ only for $p \in 2 \N + 1$ or $u^k$
for an integer  $k \geq 2$).
See \cite{LO} for details.
In the cubic case, we proved pathwise global well-posedness.
For an odd integer $p \geq 5$, 
we also used an invariant measure argument to prove almost sure global well-posedness
with suitable random initial data.\footnote{Strictly speaking, 
almost sure global well-posedness holds for the noise
$\sqrt 2 D^\frac 12\xi$, which makes the Gibbs measure for the standard NLW
invariant under the SvNLW dynamics. For pathwise global well-posedness, 
a precise coefficient in front of the noise $D^\frac{1}{2} \xi$ does not play any role.}
In this paper, our goal is to investigate further well-posedness of SvNLW \eqref{SvNLW} with an additive forcing $D^\al \xi$
and, in particular, 
prove
 pathwise global well-posedness for any $p > 1$, where the range of $\al < \frac 12$ depends on the degree $p >1$ of the nonlinearity.

%
%
%

\subsection{SvNLW with an additive stochastic forcing}
\label{SUBSEC:Sv}

Let us first consider SvNLW \eqref{SvNLW}.
 We say that $u$ is a solution to SvNLW \eqref{SvNLW} if $u$ satisfies the following Duhamel formulation of \eqref{SvNLW}:
\begin{align}
u(t) = V(t) (u_0, u_1) - \int_0^t S(t-t') \big( |u|^{p-1} u \big) (t') dt' + \Psi.
\label{Duh}
\end{align}

\noi
Here, $V(t)$ is the linear propagator defined by
\begin{equation}
V(t) (u_0, u_1) = e^{-\frac{D}{2} t} \bigg( \cos(t\jbb{D}) + \frac{D}{2\jbb{D}} \sin(t \jbb{D}) \bigg) u_0 + e^{-\frac{D}{2}t} \frac{\sin(t\jbb{D})}{\jbb{D}} u_1
\label{defV}
\end{equation}

\noi
and 
$S(t)$ is defined by
\begin{equation}
S(t) = e^{-\frac{D}{2}t} \frac{\sin(t \jbb{D})}{\jbb{D}},
\label{defS}
\end{equation}

\noi
where
\[ \jbb{D} = \sqrt{1-\tfrac{3}{4}\Dl}, \]

\noi
and $\Psi$ denotes the stochastic convolution defined by 
\begin{equation}
\Psi := \Psi_\al = \int_0^t S(t-t') D^\al \xi (dt').
\label{Psi}
\end{equation}

\noi
A standard argument shows that 
$\Psi$ belongs to $C(\R_+; W^{\frac 12 - \al - \eps, \infty} (\T^2))$ almost surely, where $\eps > 0$ is sufficiently small; see Lemma \ref{LEM:regPsi} below. 
In particular, when  $\al < \frac 12$, $\Psi$ is a well-defined function on $\R_+ \times \T^2$.

We first state a local well-posedness result for SvNLW \eqref{SvNLW}.

\begin{theorem}\label{THM:LWP1}
Let $p > 1$ and $\al < \frac 12$. Define $q$, $r$, and $\s$ as follows.
\begin{itemize}
\item[(i)] When $1 < p < 2$, set $q = 2+\dl$, $r = \frac{4+2\dl}{1+\dl}$, and $\s = 0$, for some sufficiently small $\dl > 0$.
\item[(ii)] When $p \geq 2$, set $q = p + \dl$, $r = 2p$, and $\s = 1 - \frac{1}{p+\dl} - \frac 1p$ for some arbitrary $\dl > 0$.
\end{itemize}

\noi
Let $s \geq \s$. Then, SvNLW \eqref{SvNLW} is pathwise locally well-posed in $\H^s(\T^2)$. More precisely, given any $(u_0, u_1) \in \H^s(\T^2)$, there exists $T = T_\o(u_0, u_1)$ (which is positive almost surely) and a unique solution $u$ to \eqref{SvNLW} with $(u, \dt u)|_{t=0} = (u_0, u_1)$ in the class
\[ \Psi + C([0,T]; H^\s(\T^2)) \cap L^q([0,T]; L^r(\T^2)). \]
\end{theorem}

We present the proof of Theorem \ref{THM:LWP1} in Section \ref{SEC:LWP1}. The proof of Theorem \ref{THM:LWP1} is based on the following first order expansion (\cite{McKean, BO96, DPD}):
\begin{equation}
u = v + \Psi,
\label{exp1}
\end{equation}

\noi
where the residual term $v$ satisfies the following equation:
\begin{equation}
\begin{cases}
\dt^2 v + (1-\Dl) v + D\dt v + |v + \Psi|^{p-1}(v + \Psi) = 0 \\
(v, \dt v) |_{t=0} = (u_0, u_1).
\end{cases}
\label{SvNLW_res}
\end{equation}

\noi
See Proposition \ref{PROP:LWP1} for the pathwise local well-posedness result at the level of the residual term $v$ using the homogeneous Strichartz estimates for the viscous wave equation (Lemma \ref{LEM:hom_str}).

The main idea of the proof of pathwise local well-posedness of SvNLW \eqref{SvNLW} comes from \cite{CKO}. Note that the nonlinearity $|u|^{p-1}u$ in SvNLW \eqref{SvNLW} is not necessarily algebraic for general $p>1$, which creates a difficulty for obtaining the difference estimate when applying the contraction argument. To deal with this issue, we apply the idea from Oh-Okamoto-Pocovnicu \cite{OOP} using the fundamental theorem of calculus.

\begin{remark} \rm
(i) Using the same argument, the proof of Theorem \ref{THM:LWP1} works for both the defocusing case (with the nonlinearity $|u|^{p-1}u$) and the focusing case (with the nonlinearity $-|u|^{p-1}u$, i.e. with the negative sign).

The proof of Theorem \ref{THM:LWP1} also works for SvNLW with nonlinearity $u^k$, where $k \geq 2$ is an integer. In fact, a simple argument based on Sobolev's inequality can be applied to prove local well-posedness of SvNLW with nonlinearity $u^k$ in the class $\Psi + C([0,T]; \H^s(\T^2))$ for $s \geq 1$. See, for example, Proposition 3.1 in \cite{LO}.

\smallskip \noi
(ii) As it is written in Theorem \ref{THM:LWP1}, we point out that the regularity of initial data can be lowered to the subcritical case, i.e. $s \geq s_{\text{crit}}$ (with a strict inequality when $p = 3$), where $s_{\text{crit}}$ is the critical regularity as defined in \eqref{scrit1} (note that $s_{\text{crit}} \leq \s$ with $\s$ defined in Theorem \ref{THM:LWP1}). See Theorem \ref{THM:LWP3} and Remark \ref{REM:LWP3} for details.

\smallskip \noi
(iii) One can also directly prove local well-posedness of \eqref{SvNLW} for $u \in L^q ([0, T]; L^r (\T^2))$ for some appropriate $q, r \geq 2$. Specifically, in the Duhamel formulation \eqref{Duh}, the linear term $V(t) (u_0, u_1)$ can be estimated by the Strichartz estimate (Lemma \ref{LEM:hom_str}), the nonlinear perturbation term $\int_0^t S(t - t') (|u|^{p-1} u) (t') dt'$ can be estimated by the Schauder estimate (Lemma \ref{LEM:Sch}) along with Young's convolution inequality, and the stochastic convolution $\Psi$ can also be bounded in $L^q ([0, T]; L^r (\T^2))$ (Lemma \ref{LEM:regPsi}). This approach yields a stronger uniqueness result since the solution does not depend on any specific structure such as \eqref{exp1}. 

Nevertheless, this paper is meant to be a continuation of the work in \cite{LO}, and so we choose to study \eqref{SvNLW} from a dispersive point of view. Due to the assumption that the initial data lies in the Sobolev space $\H^s (\T^2)$ for some $s \geq 0$, it is more natural to construct the solution in $C([0, T]; H^s (\T^2))$ for $T > 0$. The Strichartz spaces $L^q ([0, T]; L^r (\T^2))$ can be viewed as ``helper'' spaces that allow us to show local well-posedness for rough initial data (i.e.~with $s \geq 0$ as small as possible).
\end{remark}

\medskip

We now turn our attention to pathwise global well-posedness of SvNLW \eqref{SvNLW}, and we restrict our attention to the defocusing case. Our pathwise global well-posedness result reads as follows.

\begin{theorem}\label{THM:GWP1}
Let $p > 1$ and $\al < \min(\frac 12, \frac{2}{p-1} - \frac 12)$. Let $\s = \max (0, 1 - \frac{1}{p+\dl} - \frac 1p)$ for some arbitrary $\dl > 0$ and let $s \geq \s$. Then, SvNLW \eqref{SvNLW} is pathwise globally well-posed in $\H^s(\T^2)$. More precisely, given any $(u_0, u_1) \in \H^s(\T^s)$, there exists a unique global-in-time solution $u$ to \eqref{SvNLW} with $(u, \dt u)|_{t=0} = (u_0, u_1)$ in the class
\[ \Psi + C(\R_+; H^\s(\T^2)). \]
\end{theorem}

In Theorem \ref{THM:GWP1}, the uniqueness holds in the following sense. For any $t_0 \in \R_+$, there exists a time interval $I(t_0) \ni t_0$ such that the solution $u$ to \eqref{SvNLW} is unique in
\[ \Psi + C(I(t_0); H^\s(\T^2)) \cap L^q(I(t_0); L^r(\T^2)), \]
where $q,r \geq 2$ are as in Theorem \ref{THM:LWP1}.

As stated in Theorem \ref{THM:GWP1}, when $1 < p \leq 3$, we have the condition $\al < \frac 12$; when $p > 3$, we have the condition $\al < \frac{2}{p-1} - \frac 12$. As $p \to \infty$, the condition for $\al$ becomes $\al \leq -\frac 12$. Note that when $1 < p < 5$, we can prove pathwise global well-posedness of SvNLW \eqref{SvNLW} with the space-time white noise (i.e. $\al = 0$).

We prove Theorem \ref{THM:GWP1} by studying \eqref{SvNLW_res} for the residual term $v$ in Section \ref{SEC:GWP1}. From the proof of Theorem \ref{THM:LWP1}, we see that pathwise global well-posedness follows once we control the $\H^1$-norm of $\vec v(t) := (v(t), \dt v (t))$. For this purpose, we study the evolution of the energy
\begin{equation}
E(\vec v) = \frac 12 \int_{\T^2} (v^2 + |\nb v|^2) dx + \frac 12 \int_{\T^2} (\dt v)^2 dx + \frac{1}{p+1} \int_{\T^2} |v|^{p+1} dx,
\label{defE}
\end{equation}

\noi
which is conserved under the (deterministic) usual NLW:
\[ \dt^2 u + (1-\Dl)u + |u|^{p-1}u = 0. \]

\noi
Note that for our problem, we proceed with the first order expansion \eqref{exp1}, where the residual term $v = \Psi - u$ only satisfies \eqref{SvNLW_res}. In this case, the energy $E(\vec v)$ is not conserved under the equation \eqref{SvNLW_res} because of the perturbative term $|v+\Psi|^{p-1}(v+\Psi) - |v|^{p-1}v$. For our problem, we first follow the globalization argument by Burq-Tzvetkov \cite{BT14} and establish an exponential growth bound on $E(\vec v)$, which works in the sub-cubic case $1 < p \leq 3$. For the super-cubic case $p > 3$, this argument no longer works due to the high homogeneity of the non-linearity. When $3 < p \leq 5$, we use an integration by parts trick introduced by Oh-Pocovnicu \cite{OP16}. In the super-quintic case $p > 5$, we use a trick involving the Taylor expansion, where the idea comes from Latocca \cite{Lat}.

One important prerequisite for studying the evolution of the energy $E(\vec v)$ is that the local-in-time solution $\vec v$ lies in $\H^1(\T^2)$, which is not guaranteed by the pathwise local well-posedness result (Theorem \ref{THM:LWP1}) as it is written. Nonetheless, due to the dissipative nature of the equation, we show that $\vec v(t)$ indeed belongs to $\H^1(\T^2)$ for any $t > 0$ by using the Schauder estimate (Lemma \ref{LEM:Sch}) along with Theorem \ref{THM:LWP1}. See Section \ref{SEC:GWP1} for details.

\medskip
We conclude our introduction by stating several remarks.

\begin{remark} \rm
(i) We point out that Theorem \ref{THM:LWP1} and Theorem \ref{THM:GWP1} also hold if we have $-\Dl$ instead of $1-\Dl$ in \eqref{SvNLW} by using an essentially identical proof.

\smallskip \noi
(ii) In \cite{LO}, T. Oh and the author studied SvNLW \eqref{SvNLW} with $\al = \frac 12$. In this case, due to $\al = \frac 12$, the stochastic term $\Psi$ defined in \eqref{Psi} turns out to be merely a distribution, so that we studied a renormalized version of \eqref{SvNLW} and proved pathwise global well-posedness in the cubic case. Because of the singular nature of the stochastic convolution in this setting, the standard Gronwall argument does not work, and so we used a Yudovich-type argument to bound the corresponding energy.

In the same paper, we also proved almost sure global well-posedness of \eqref{SvNLW} with $p \in 2\N+1$ and with random initial data, using the formal invariance of the Gibbs measure. However, the argument only works for $\al = \frac 12$, so it does not apply to our problem with $\al < \frac 12$ in this paper. Instead, in this paper, we establish pathwise global well-posedness of SvNLW \eqref{SvNLW}.

\smallskip \noi
(iii) We can also consider the vNLW with randomized initial data:
\begin{align}
\begin{cases}
\dt^2 u + (1 - \Dl) u + D \dt u + |u|^{p-1} u = 0 \\
(u, \dt u) |_{t = 0} = (u_0^\o, u_1^\o).
\end{cases}
\label{vNLWr}
\end{align}

\noi
Here, the randomization $(u_0^\o, u_1^\o)$ of the initial data $(u_0, u_1)$ is defined by
\begin{align}
(u_0^\o, u_1^\o) := \bigg( \sum_{n \in \Z^2} g_{n, 0} (\o) \ft{u_0} (n) e^{in \cdot x}, \sum_{n \in \Z^2} g_{n, 1} (\o) \ft{u_1} (n) e^{in \cdot x} \bigg),
\label{rand}
\end{align}
where for $j = 0, 1$, $\ft{u_j} (-n) = \cj{\ft{u_j} (n)}$ for all $n \in \Z^2$ and $\{g_{n, j}\}_{n \in \Z^2}$ is a sequence of mean zero complex-valued random variables such that $g_{-n, j} = \cj{g_{n, j}}$ for all $n \in \Z^2$. Moreover, we assume that $g_{0, j}$ is real-valued for $j = 0, 1$, $\{ g_{0, j}, \Re g_{n, j}, \Im g_{n, j} \}_{n \in \I, j = 0, 1}$ are independent with $\I = (\Z_+ \times \{0\}) \cup (\Z \times \Z_+)$, and there exists a constant $c > 0$ such that on the probability distributions $\mu_{n, j}$ of $g_{n, j}$, we have
\begin{align}
\int e^{\g \cdot x} d\mu_{n, j} (x) \leq e^{c |\g|^2}, \quad j = 0, 1
\label{exp_bdd}
\end{align}
for all $\g \in \R^2$ when $n \in \Z^2 \setminus \{0\}$ and all $\g \in \R$ when $n = 0$. Note that  \eqref{exp_bdd} is satisfied for standard complex-valued Gaussian random variables, standard Bernoulli random variables, and any random variables with compactly supported distributions.

The randomization \eqref{rand} allows us to consider almost sure global well-posedness of \eqref{vNLWr} for $(u_0, u_1)$ living in negative Sobolev spaces. For almost sure local well-posedness, we consider the following first order expansion similar to \eqref{exp1}:
\begin{align*}
u = v + z,
\end{align*}
where $z$ is the solution of the linear viscous wave equation with initial data $(u_0^\o, u_1^\o)$:
\begin{align*}
z(t) = z^\o (t) := V(t) (u_0^\o, u_1^\o)
\end{align*}
with $V(t)$ defined as in \eqref{defV}. By using the Schauder estimate (Lemma \ref{LEM:Sch}), we can establish similar (but stronger) probabilistic Strichartz estimates for $z$ and $\jb{\nb}^{-1} \dt z$ as in \cite{OP16, OP17}. This enables us to prove almost sure local well-posedness of \eqref{vNLWr} using a similar argument as for proving Theorem \ref{THM:LWP1}, as long as $(u_0, u_1) \in \H^s (\T^2)$ with $s > - \frac{1}{p}$. On the other hand, the proof for almost sure global well-posedness of \eqref{vNLWr} is much simpler than that for Theorem \ref{THM:GWP1}, since $z(t)$ is smooth for $t > 0$ thanks to the parabolic smoothing. We omit details since this is not the main focus in this paper. 
\end{remark}

\section{Preliminary lemmas}
\label{SEC:2}

In this section, we discuss some notations and lemmas that are necessary for proving our well-posedness results.

We use $A \les B$ to denote $A \leq CB$ for some constant $C>0$, and we write $A \sim B$ if $A \les B$ and $B \les A$. Also, we use $a+$ (and $a-$) to denote $a+\eps$ (and $a-\eps$, respectively) for arbitrarily small $\eps > 0$. In addition, we use short-hand notations to work with space-time function spaces. For example, $C_T H_x^s = C([0,T]; H^s(\T^d))$.

\subsection{Sobolev spaces and Besov spaces}
Let $s \in \R$. We denote $H^s(\T^d)$ as the $L^2$-based Sobolev space with the norm:
\[ \| u \|_{H^s(\T^d)} = \| \jb{n}^s \ft u (n) \|_{\l_n^2(\Z^d)}, \]
where $\ft u (n)$ is the Fourier coefficient of $u$ and $\jb{\cdot} = (1 + |\cdot|)^{\frac 12}$. We then define $\H^s(\T^d)$ as
\[ \H^s(\T^d) = H^s(\T^d) \times H^{s-1}(\T^d). \]
Also, we denote $W^{s,p}(\T^d)$ as the $L^p$-based Sobolev space with the norm:
\[ \| u \|_{W^{s,p}(\T^d)} = \big\| \F^{-1} ( \jb{n}^s \ft u (n) ) \big\|_{L^p(\T^d)}, \]
where $\F^{-1}$ denotes the inverse Fourier transform on $\T^d$. When $p=2$, we have $H^s(\T^d) = W^{s,2}(\T^d)$.

Let $\varphi : \R \to [0,1]$ be a bump function such that $\varphi \in C_c([-\tfrac 85, \tfrac 85])$ and $\varphi \equiv 1$ on $[-\tfrac 54, \tfrac 54]$. For $\xi \in \R^d$, we define $\varphi_0(\xi) = \varphi(|\xi|)$ and
\[ \varphi_j(\xi) = \varphi\big( \tfrac{|\xi|}{2^j} \big) - \varphi\big( \tfrac{|\xi|}{2^{j-1}} \big) \]
for $j \in \Z_+$. Note that
\begin{equation}
\sum_{j \in \Z_{\geq 0}} \varphi_j(\xi) = 1
\label{phi_sum}
\end{equation}

\noi
for any $\xi \in \R^d$. For $j \in \Z_{\geq 0}$, we define the Littlewood-Paley projector $\P_j$ as
\[ \P_j u = \F^{-1} (\varphi_j \ft u). \]
Due to \eqref{phi_sum}, we have
\begin{equation}
u = \sum_{j=0}^\infty \P_j u.
\label{decomp}
\end{equation}

We also recall the definition of Besov spaces $B_{p,q}^s (\T^d)$ equipped with the norm:
\[ \| u \|_{B_{p,q}^s(\T^d)} = \big\| 2^{sj} \| \P_j u \|_{L_x^p(\T^d)} \big\|_{\l_j^q(\Z_{\geq 0})}. \]
Note that $H^s(\T^d) = B_{2,2}^s(\T^d)$.

We then recall the definition of paraproducts introduced by Bony \cite{Bony}. For details, see \cite{BCD, GIP}. For given functions $u$ and $v$ on $\T^d$ of regularities $s_1$ and $s_2$, respectively. By \eqref{decomp}, we can write the product $uv$ as
\begin{align*}
uv &= u \pl u + u \pe v + u \pg v \\
:\!\! &= \sum_{j < k-2} \P_j u \P_k v + \sum_{|j-k| \leq 2} \P_j u \P_k v + \sum_{k < j-2} \P_j u \P_k v.
\end{align*}

\noi
The term $u \pl v$ (and the term $u \pg v$) is called the paraproduct of $v$ by $u$ (and the paraproduct of $u$ by $v$, respectively), and it is well defined as a distribution of regularity $\min(s_2, s_1+s_2)$ (and $\min(s_1, s_1+s_2)$, respectively). The term $u \pe v$ is called the resonant product of $u$ and $v$, and it is well defined in general only if $s_1+s_2 > 0$.

With these definitions in hand, we recall some basic properties of Besov spaces.

\begin{lemma}\label{LEM:Bes}
\textup{(i)} Let $s_1,s_2 \in \R$ and $1 \leq p, p_1, p_2, q \leq \infty$ which satisfies $\frac{1}{p} = \frac{1}{p_1} + \frac{1}{p_2}$. Then, we have
\begin{equation}
\| u \pl v \|_{B_{p,q}^{s_2}(\T^d)} \les \| u \|_{L^{p_1}(\T^d)} \| v \|_{B_{p_2,q}^{s_2}(\T^d)}.
\label{Bes_pl}
\end{equation}

\noi
When $s_1 + s_2 > 0$, we have
\begin{equation}
\| u \pe v \|_{B_{p,q}^{s_1+s_2}(\T^d)} \les \| u \|_{B_{p_1,q}^{s_1}(\T^d)} \| v \|_{B_{p_2,q}^{s_2}(\T^d)}.
\label{Bes_pe}
\end{equation}

\smallskip
\noi
\textup{(ii)} Let $s_1 < s_2$ and $1\leq p,q \leq \infty$. Then, we have
\begin{equation}
\| u \|_{B_{p,q}^{s_1}(\T^d)} \les \| u \|_{W^{s_2,p}(\T^d)}.
\label{Bes_emb1}
\end{equation}

\noi
In particular, when $q = \infty$, we have
\begin{equation}
\| u \|_{B_{p,\infty}^{s_1}(\T^d)} \les \| u \|_{W^{s_1,p}(\T^d)}.
\label{Bes_emb2}
\end{equation}
\end{lemma}

See \cite{BCD, MW} for the proofs of \eqref{Bes_pl} and \eqref{Bes_pe} in the $\R^d$ setting, which can be easily extended to the $\T^d$ setting. The embedding \eqref{Bes_emb1}  follows from the $L^p$ boundedness of $\P_j$ and the $\l^q$-summability of $\big\{ 2^{(s_1-s_2)j} \big\}_{j \in \Z_{\geq0}}$, and the embedding \eqref{Bes_emb2} follows easily from the $L^p$ boundedness of $\P_j$.

Using \eqref{Bes_pl} and \eqref{Bes_pe}, we get the following product estimate.

\begin{corollary}\label{COR:bilin}
Let $s > 0$, $1 \leq p,q \leq \infty$ and $1 \leq p_1,p_2,q_1,q_2 \leq \infty$ satisfying
\[ \frac{1}{p_1} + \frac{1}{q_1} = \frac{1}{p_2} + \frac{1}{q_2} = \frac{1}{p}. \]
Then, \[ \| uv \|_{B_{p,q}^s(\T^d)} \les \| u \|_{B_{p_1,q}^s(\T^d)} \| v \|_{L^{q_1}(\T^d)} + \| u \|_{L^{p_2}(\T^d)} \| v \|_{B_{q_2,q}^s(\T^d)}. \]
\end{corollary}

\medskip
Next, we recall the following chain rule estimates.

\begin{lemma}\label{LEM:chain}
Let $u$ be a smooth function on $\T^d$, $s \in (0,1)$, $r \geq 2$. Let $F$ denote the function $F(u) = |u|^{r-1}u$ or $F(u) = |u|^r$. 

\smallskip \noi
\textup{(i)} Let $1 < p, p_1 < \infty$ and $1 < p_2 \leq \infty$ satisfying $\frac{1}{p} = \frac{1}{p_1} + \frac{1}{p_2}$. Then, we have
\begin{equation}
\| F(u) \|_{W^{s,p}(\T^d)} \les \| u \|_{W^{s,  p_1}(\T^d)} \big\| |u|^{r-1} \big\|_{L^{p_2}(\T^d)}.
\label{chain1}
\end{equation}

\smallskip \noi
\textup{(ii)} Let $1 \leq p,q \leq \infty$ and $1 \leq p_1,p_2 \leq \infty$ satisfying $\frac{1}{p} = \frac{1}{p_1} + \frac{1}{p_2}$. Then, we have
\begin{equation}
\| F(u) \|_{B_{p,q}^s(\T^d)} \les \| u \|_{B_{p_1,q}^s(\T^d)} \big\| |u|^{r-1} \big\|_{L^{p_2}(\T^d)}.
\label{chain2}
\end{equation}
\end{lemma}

The estimate \eqref{chain1} follows immediately from the fractional chain rule on $\T^d$ in \cite{Gat}. For the proof of \eqref{chain2}, see, for example, Lemma 3.5 in \cite{Lat} in the $\R^d$ setting, which can be easily extended to the $\T^d$ setting.

\medskip
Lastly, we recall the following Gagliardo-Nirenberg interpolation inequality.

\begin{lemma}\label{LEM:Gag}
Let $p_1, p_2 \in (1,\infty)$ and $s_1,s_2 > 0$. Let $p > 1$ and $\ta \in (0,1)$ satisfying
\[ -\frac{s_1}{d} + \frac{1}{p} = (1-\ta)\bigg( \frac{1}{p_1} - \frac{s_2}{d} \bigg) + \frac{\ta}{p_2} \quad \text{and} \quad s_1 \leq (1-\ta)s_2. \]
Then, for $u \in W^{s_2,p_1}(\T^d) \cap L^{p_2}(\T^d)$, we have
\[ \| u \|_{W^{s_1,p}(\T^d)} \les \| u \|_{W^{s_2,p_1}(\T^d)}^{1-\ta} \| u \|_{L^{p_2}(\T^d)}^\ta. \]
\end{lemma}

This inequality follows from a direct application of Sobolev's inequality on $\T^d$ (see \cite{BO}) and then interpolation.

\subsection{On the stochastic term}
\label{SEC:2.2}

In this subsection, we discuss the regularity properties of the stochastic term $\Psi$ defined in \eqref{Psi}. Given $N \in \N$, we denote $\Psi_N = \pi_N \Psi$ as the truncated stochastic convolution, where $\pi_N$ is the frequency cutoff onto the spatial frequencies $\{|n| \leq N\}$. Then, we have the following regularity result for $\Psi$.
\begin{lemma}\label{LEM:regPsi}
For any $\eps > 0$ and $T > 0$, 
$\Psi_N$ converges to $\Psi$
in $C([0,T];W^{1-2\al -\eps,\infty}(\T^2))$ almost surely.
 In particular, we have
  \[\Psi \in C([0,T];W^{\frac 12-\al -\eps,\infty}(\T^2))
  \]
  
  \noi
  almost surely.
\end{lemma}

The proof of Lemma \ref{LEM:regPsi} follows from a straightforward modification of the proof in \cite[Lemma 3.1]{GKO2}, and so we omit details. See also \cite[Proposition 2.1]{GKO}.

\begin{remark}\label{REM:dPsi}\rm
One can use an integration by parts to write
\[ \ft \Psi (t, n) = -\int_0^t B_n(t') \frac{d}{ds} \Big|_{s=t'} \Big( e^{- \frac{|n|}{2}(t-s)} \frac{\sin ((t-s)\jbb{n} )}{\jbb{n}} |n|^\al \Big) dt'  \]
almost surely, which allows us to compute that
\begin{align*}
\dt \ft \Psi (t, n) &= \int_0^t \bigg( -\frac{|n|}{2 \jbb{n}} e^{-\frac{|n|}{2} (t-t')} \sin((t-t')\jbb{n}) \\
&\quad + e^{-\frac{|n|}{2} (t-t')} \cos((t-t')\jbb{n}) \bigg) |n|^\al dB_n(t')
\end{align*}

\noi
almost surely. Using a similar argument as in \cite[Lemma 3.1]{GKO2} or \cite[Proposition 2.1]{GKO}, we have $\dt \Psi \in C([0,T]; W^{-\frac 12 - \al -, \infty} (\T^2))$ almost surely. This will be useful in the proof of pathwise global well-posedness of SvNLW \eqref{SvNLW} in Subsection \ref{SUBSEC:GWP2}.
\end{remark}

\subsection{Linear estimates}
In this subsection, we show some relevant linear estimates and the Strichartz estimates that are used to prove our well-posedness results.

Let
\begin{equation*}
P(t) = e^{-\frac{D}{2} t}
\end{equation*}
be the Poisson kernel with a parameter $\frac{t}{2}$, which appears in the viscous wave linear propagator $V(t)$ defined in \eqref{defV}. We first recall the following Schauder-type estimate for the Poisson kernel $P(t)$. For a proof, see Lemma 2.3 in \cite{LO}.

\begin{lemma}\label{LEM:Sch}
Let $1 \leq p \leq q \leq \infty$ and $\b \geq 0$. Then, we have
\[ \| D^\b P(t) \phi \|_{L^q(\T^d)} \les t^{-\b-d(\frac{1}{p}-\frac{1}{q})} \| \phi \|_{L^p(\T^d)} \]

\noi
for any $0 < t \leq 1$. 
\end{lemma}

Next, we turn our attention to the Strichartz estimates for the homogeneous linear viscous wave equation on $\T^d$. We recall that the linear propagator $V(t)$ is defined in \eqref{defV}.

\begin{lemma}\label{LEM:hom_str}
Given $s \geq 0$, suppose that $2 < q \leq \infty$, $2 \leq r \leq \infty$ satisfy the following scaling condition:
\begin{equation}
\frac{1}{q} + \frac{d}{r} = \frac{d}{2} - s.
\label{hom_cond}
\end{equation}

\noi
Then, we have
\begin{equation}
\| V(t) (\phi_0, \phi_1) \|_{C ([0,T]; \H^{s-1}(\T^d))} \les \| (\phi_0, \phi_1) \|_{\H^s(\T^d)}
\label{hom_Hs}
\end{equation}
and
\begin{equation}
\| V(t) (\phi_0, \phi_1) \|_{L^q ([0,T]; L_x^r(\T^d))} \les \| (\phi_0, \phi_1) \|_{\H^s(\T^d)}
\label{hom_LqLr}
\end{equation}
for all $0 < T \leq 1$.
\end{lemma}

\begin{proof}
The bound \eqref{hom_Hs} can be immediately seen from the definition of the $\H^s$-norm and the fact that $e^{-\frac{|n|}{2} t} \leq 1$, $|\cos (t \jbb{n})| \leq 1$, and $|\sin (t \jbb{n})| \leq 1$.

To prove \eqref{hom_LqLr}, we use the $TT^*$ method. We first consider the case when $s = 0$. Let
\[ V_1(t) = e^{-\frac{D}{2}t} \cos(t\jbb{D}), \qquad V_2(t) = e^{-\frac{D}{2}t} \frac{D}{2\jbb{D}} \sin(t\jbb{D}), \]
so that
\[ V(t)(\phi_0, \phi_1) = V_1(t) \phi_0 + V_2(t) \phi_0 + S(t) \phi_1. \]
Let $L: L^2(\T^d) \to L_T^q L_x^r(\T^d) $ be the linear operator given by $L \phi = V_1(t) \phi$. Note that $L^*$ is the linear operator given by 
\[ L^* f = \int_0^T V_1(t') f(t') dt'  \]
for any space-time Schwartz function $f$. By Minkowski's integral inequality, the Schauder estimate (Lemma \ref{LEM:Sch}) twice, the scaling condition \eqref{hom_cond}, and the Hardy-Littlewood-Sobolev inequality, we have
\begin{align*}
\| LL^* f \|_{L_T^q L_x^r} &\leq \bigg\| \int_0^T \big\| e^{-\frac{D}{2} (t+t')} \cos(t \jbb{D}) \cos(t' \jbb{D}) f(t') \big\|_{L_x^r} dt' \bigg\|_{L_T^q} \\
&\les \bigg\| \int_0^T \frac{1}{|t-t'|^{2(\frac 12 - \frac{1}{r})}} \big\| e^{-\frac{D}{4} (t+t')} \cos(t \jbb{D}) \cos(t' \jbb{D}) f(t') \big\|_{L_x^2} dt' \bigg\|_{L_T^q} \\
&\les \bigg\| \int_0^T \frac{1}{|t-t'|^{2(\frac{1}{r'} - \frac{1}{r})}} \| f(t') \|_{L_x^{r'}} dt' \bigg\|_{L_T^q} \\
&= \bigg\| \int_0^\infty \frac{1}{|t-t'|^{2/q}} \| \ind_{[0,T]}f(t') \|_{L_x^{r'}} dt' \bigg\|_{L_T^q} \\
&\les \| f \|_{L_T^{q'}\! L_x^{r'}}.
\end{align*}

\noi
Thus, by a standard duality argument, we obtain
\[ \| V_1(t) \phi_0 \|_{L_T^q L_x^r(\T^d)} \les \| \phi_0 \|_{L^2(\T^d)}. \]
By using similar arguments, we obtain
\[ \| V_2(t) \phi_0 \|_{L_T^q L_x^r(\T^d)} \les \| \phi_0 \|_{L^2}, \qquad \| S(t) \phi_1 \|_{L_T^q L_x^r} \les \| \phi_0 \|_{H^{-1}(\T^d)}, \]
so that we have
\begin{equation}
\| V(t) (\phi_0, \phi_1) \|_{L_T^q L_x^r(\T^d)} \les \| (\phi_0, \phi_1) \|_{\H^0(\T^d)}.
\label{hom1}
\end{equation}

When $s > 0$, by Sobolev's inequality, the scaling condition \eqref{hom_cond}, and \eqref{hom1}, we obtain
\begin{align*}
\| V(t) (\phi_0, \phi_1) \|_{L_T^q L_x^r} &\les \| V(t) (\jb{\nb}^s \phi_0, \jb{\nb}^s \phi_1) \|_{L_T^q L_x^{1/(\frac{s}{d} + \frac{1}{r})}} \\
&= \| V(t) (\jb{\nb}^s \phi_0, \jb{\nb}^s \phi_1) \|_{L_T^q L_x^{d / (\frac{d}{2} - \frac{1}{q})}} \\
&\les \| (\jb{\nb}^s \phi_0, \jb{\nb}^s \phi_1) \|_{\H^0} \\
&= \| (\phi_0, \phi_1) \|_{\H^s},
\end{align*}
as desired.
\end{proof}

\begin{remark} \rm
(i) Compared to the Strichartz estimates for the usual linear wave equations (\cite{GV,LS,KT,GKO}), the Strichartz estimates for the homogeneous linear viscous wave equation on $\T^d$ hold for a larger class of pairs $(q,r)$, thanks to the parabolic smoothing effect.

\smallskip \noi
(ii) In \cite{KC1}, Kuan-\v{C}ani\'{c} proved the Strichartz estimates for the homogeneous linear viscous wave equation on $\R^d$. They used the method from Keel-Tao \cite{KT}, so that their result requires $(q,r)$ to be $\s$-admissible for some $\s > 0$, i.e. $(q,r,\s) \neq (2,\infty,1)$ and 
\[ \frac{2}{q} + \frac{2\s}{r} \leq \s. \]
We point out that the $TT^*$ method we use in the proof also works on $\R^d$ and does not have this $\s$-admissible restriction on $q$ and $r$. However, our proof works only for $s \geq 0$.
\end{remark}

We complete this subsection by establishing the following inhomogeneous linear estimates.

\begin{lemma}\label{LEM:lin1}
Let $p \geq 2$ and let $S(t)$ be as in \eqref{defS}. Then, given $\dl > 0$, we have
\begin{equation}
\bigg\| \int_0^t S(t-t') F(t') dt' \bigg\|_{L^{p+\dl} ([0,T]; L_x^{2p}(\T^2))} \les \| F \|_{L^1 ([0,T]; L^2(\T^2))}
\label{lin1}
\end{equation}
for any $0 < T \leq 1$.
\end{lemma}

\begin{proof}
We let
\[ s = 1 - \frac{1}{p + \dl} - \frac{2}{2p} = 1 - \frac{1}{p + \dl} - \frac{1}{p}, \]
so that $(p+\dl, 2p, s)$ satisfies the scaling condition in Lemma \ref{LEM:hom_str}. By Minkowski's integral inequality and Lemma \ref{LEM:hom_str}, we obtain
 \begin{align*}
\bigg\| \int_0^t S(t-t') F(t') dt' \bigg\|_{L_T^{p+\dl} L_x^{2p}} &\les \int_0^T \| \ind_{[0,t]}(t') S(t-t')F(t') \|_{L_T^{p+\dl} L_x^{2p}} dt' \\
&\les \int_0^T \| F(t') \|_{H_x^{s-1}} dt' \\
&\leq \| F \|_{L_T^1 L_x^2},
\end{align*}
so that \eqref{lin1} follows.
\end{proof}

\begin{lemma}\label{LEM:lin2}
Let $S(t)$ be as in \eqref{defS}. Then, given $s \leq 1$, we have
\begin{equation}
\bigg\| \int_0^t S(t-t') F(t') dt' \bigg\|_{C([0,T]; H_x^s(\T^2))} \les \| F \|_{L^1 ([0,T]; L_x^2(\T^2))},
\label{lin3}
\end{equation}
\begin{equation}
\bigg\| \dt \int_0^t S(t-t') F(t') dt' \bigg\|_{C([0,T]; H_x^{s-1}(\T^2))} \les \| F \|_{L^1 ([0,T]; L_x^2(\T^2))}
\label{lin4}
\end{equation}
for any $0 < T \leq 1$.
\end{lemma}

\begin{proof}
The estimate \eqref{lin3} follows from \eqref{defS} and Minkowski's integral inequality. The estimate \eqref{lin4} follows similarlly by noting that
\[ \dt \int_0^t S(t-t') F(t') dt' = \int_0^t \dt S(t-t') F(t') dt', \]
where
\[ \dt S(t) = e^{-\frac{D}{2} t} \bigg( \cos(t \jbb{D}) - \frac{D}{2 \jbb{D}} \sin(t \jbb{D}) \bigg). \]
\end{proof}

\section{Local well-posedness of SvNLW}
\label{SEC:LWP1}
In this section, we prove Theorem \ref{THM:LWP1}, pathwise local well-posedness for SvNLW \eqref{SvNLW}. As mentioned in Subsection \ref{SUBSEC:Sv}, we consider the following vNLW:
\begin{equation}
\begin{cases}
\dt^2 v + (1-\Dl) v + D\dt v + F(v+\Psi) = 0 \\
(v, \dt v)|_{t=0} = (u_0, u_1)
\end{cases}
\label{vNLW}
\end{equation}

\noi
for given initial data $(u_0, u_1) \in \H^s(\T^2)$, $F(u) = |u|^{p-1}u$, and $\Psi$ is the stochastic convolution defined in \eqref{Psi}. By Lemma \ref{LEM:regPsi}, we can fix a good $\o \in \O$ such that $\Psi = \Psi(\o) \in C([0,T]; W^{\frac 12 - \al - \eps, \infty}(\T^2))$ for $\al < \frac 12$ and sufficiently small $\eps > 0$, so that \eqref{vNLW} becomes a deterministic equation. Then, we have the following pathwise local well-posedness of \eqref{vNLW}.

\begin{proposition}\label{PROP:LWP1}
Let $p > 1$ and $\al < \frac 12$. Define $q$, $r$, and $\s$ as follows.
\begin{itemize}
\item[(i)] When $1 < p < 2$, set $q = 2+\dl$, $r = \frac{4+2\dl}{1+\dl}$, and $\s = 0$, for some sufficiently small $\dl > 0$.
\item[(ii)] When $p \geq 2$, set $q = p+\dl$, $r = 2p$, and $\s = 1 - \frac{1}{p+\dl} - \frac{1}{p}$, for some arbitrary $\dl > 0$.
\end{itemize}

\noi
Let $s \geq \s$. Then, $\eqref{vNLW}$ is pathwise locally well-posed in $\H^s(\T^2)$. More precisely, given any $(u_0, u_1) \in \H^s(\T^2)$, there exists $0 < T = T_\o(u_0, u_1) \leq 1$ and a unique solution $\vec v = (v, \dt v)$ to \eqref{vNLW} in the class
\[ (v, \dt v) \in C([0,T]; \H^\s (\T^2)) \quad \text{and} \quad v \in L^q([0,T]; L^r(\T^2)). \]
\end{proposition}

Note that Theorem \ref{THM:LWP1} follows immediately from Proposition \ref{PROP:LWP1}. The main idea of the proof of Proposition \ref{PROP:LWP1} comes from \cite{CKO}.

\begin{proof}
We first consider the case when $p \geq 2$. We write \eqref{vNLW} in the Duhamel formulation:
\begin{equation}
v(t) = \G(v) := V(t)(u_0,u_1) - \int_0^t S(t-t') F(v+\Psi)(t') dt',
\label{L1}
\end{equation}

\noi
where $V(t)$ and $S(t)$ are as defined in \eqref{defV} and \eqref{defS}, respectively. Let $\vec \G(v) = (\G(v), \dt\G(v))$ and $\vec v = (v, \dt v)$. Given $0 < T \leq 1$, we define the space $\mathcal{X}(T)$ as
\[ \mathcal{X}^\s(T) = \mathcal{X}_1^\s(T) \times \mathcal{X}_2^\s(T), \]
where
\begin{equation*}
\begin{split}
&\mathcal{X}_1^\s(T) := C([0,T]; H^\s (\T^2)) \cap L^{p+\dl}([0,T]; L^{2p}(\T^2)), \\
&\mathcal{X}_2^\s(T) := C([0,T]; H^{\s-1} (\T^2)).
\end{split}
\end{equation*}

\noi
Here, $\dl > 0$ is arbitrary and $\s = 1 - \frac{1}{p+\dl} - \frac 1p$. Note that this choice of $\s$ along with the $L_T^{p+\dl} L_x^{2p}$ norm satisfies the scaling condition in Lemma \ref{LEM:hom_str}. Our goal is to show that $\vec \G$ is a contraction on a ball in $\mathcal{X}^\s(T)$ for some $0 < T \leq 1$.

By \eqref{L1}, Lemma \ref{LEM:hom_str}, \eqref{defV}, Lemma \ref{LEM:lin1}, Lemma \ref{LEM:lin2}, and Sobolev's inequality with the fact that $|\T^2| = 1$, we have
\begin{equation}
\begin{split}
\| \vec \G (v) \|_{\mathcal{X}^\s(T)} &\les \| (u_0, u_1) \|_{\H^\s} + \big\| |v + \Psi|^p \big\|_{L_T^1 L_x^2} \\
&\les \| (u_0, u_1) \|_{\H^s} + T^\ta \Big( \| v \|_{L_T^{p+\dl} L_x^{2p}}^p + \| \Psi \|_{L_T^{p+\dl} L_x^{2p}}^p \Big) \\
&\les \| (u_0, u_1) \|_{\H^s} + T^\ta \Big( \| \vec v \|_{\mathcal{X}^\s(T)}^p + \| \Psi \|_{C_T W_x^{\frac 12 - \al - \eps, \infty}}^p \Big)
\end{split}
\label{SLWP1}
\end{equation}

\noi
for some $\ta > 0$ and sufficiently small $\eps > 0$. 

For the difference estimate, we use the idea from Oh-Okamoto-Pocovnicu \cite{OOP}. Noticing that $F'(u) = p|u|^{p-1}$, we use \eqref{L1}, Lemma \ref{LEM:hom_str}, Lemma \ref{LEM:lin1}, Lemma \ref{LEM:lin2}, the fundamental theorem of calculus, Minkowski's integral inequality, H\"older's inequality, and Sobolev's inequality to obtain
\begin{equation}
\begin{split}
\| \vec \G (v) - \vec \G (w) \|_{\mathcal{X}^\s(T)} &\les \| F(v+\Psi) - F(w+\Psi) \|_{L_T^1 L_x^2} \\
&= \bigg\| \int_0^1 F'(w+\Psi+\tau (v-w)) (v-w) d\tau \bigg\|_{L_T^1 L_x^2} \\
&\les \int_0^1 \| w + \Psi + \tau(v-w) \|_{L_T^p L_x^{2p}}^{p-1} \| v-w \|_{L_T^p L_x^{2p}} d\tau \\
&\les T^\ta \Big( \| v \|_{L_T^{p+\dl} L_x^{2p}}^{p-1} + \| w \|_{L_T^{p+\dl} L_x^{2p}}^{p-1} + \| \Psi \|_{L_T^{p+\dl} L_x^{2p}}^{p-1} \Big) \| v-w \|_{L_T^{p+\dl} L_x^{2p}} \\
&\les T^\ta \Big( \| \vec v \|_{\mathcal{X}^\s(T)}^{p-1} + \| \vec w \|_{\mathcal{X}^\s(T)}^{p-1} + \| \Psi \|_{C_T W_x^{\frac 12 - \al - \eps, \infty}}^{p-1} \Big) \| \vec v - \vec w \|_{\mathcal{X}^\s(T)}
\end{split}
\label{SLWP2}
\end{equation}

\noi
for some $\ta > 0$ and sufficiently small $\eps > 0$.

Thus, by choosing $T = T_\o ( \| (u_0,u_1) \|_{\H^s} ) > 0$ small enough, we obtain that $\vec \G$ is a contraction on the ball $B_R \subset \mathcal{X}^\s(T)$ of radius $R \sim 1 + \| (u_0,u_1) \|_{\H^s}$. Note that at this point, the uniqueness of the solution $v$ only holds in the ball $B_R$, but we can use a standard continuity argument to extend the uniqueness of $v$ to the entire $\mathcal{X}^\s(T)$.

For the case when $1 < p < 2$, we may have $p + \dl < 2$, so that Lemma \ref{LEM:hom_str} may not work for the $L_T^{p+\dl} L_x^{2p}$ norm. Instead, we consider the $L_T^q L_x^r$ norm with $q = 2+\dl$ and $r = \frac{4+2\dl}{1+\dl}$, where $\dl > 0$ is small enough so that $r$ is close enough to 4. We also set $\s = 0$, so that that this choice of $\s$ along with this $L_T^q L_x^r$ norm satisfies the scaling condition in Lemma \ref{LEM:hom_str}. Note that we also need to modify the definition of $\mathcal{X}_1^\s(T)$ using this $L_T^q L_x^r$ norm. We then modify \eqref{SLWP1} as follows. By \eqref{L1}, Lemma \ref{LEM:hom_str}, \eqref{defV}, Lemma \ref{LEM:lin1}, Lemma \ref{LEM:lin2}, and Sobolev's inequality, we have
\begin{align*}
\| \vec \G (v) \|_{\mathcal{X}^0(T)} &\les \| (u_0, u_1) \|_{\H^0} + \big\| |v + \Psi|^p \big\|_{L_T^1 L_x^2} \\
&\les \| (u_0, u_1) \|_{\H^s} + T^\ta \Big( \| v \|_{L_T^{2+\dl} L_x^{2p}}^p + \| \Psi \|_{L_T^{p+\dl} L_x^{2p}}^p \Big) \\
&\les \| (u_0, u_1) \|_{\H^s} + T^\ta \Big( \| \vec v \|_{\mathcal{X}^0(T)}^p + \| \Psi \|_{C_T W_x^{\frac 12 - \al - \eps, \infty}}^p \Big)
\end{align*}
for some $\ta > 0$. Here, we can ensure that $2p \leq r = \frac{4+2\dl}{1+\dl}$ for any $1 < p < 2$ by choosing $\dl = \dl(p) > 0$ small enough. A similar modification can be applied to \eqref{SLWP2} to obtain a difference estimate, which then allows us to close the contraction argument.
\end{proof}

\begin{remark}\label{REM:LWP1}\rm
We point out that the local well-posedness result of vNLW \eqref{vNLW} can be improved using the inhomogeneous Strichartz estimates. In particular, we can show that \eqref{vNLW} is locally well-posed in $\H^s(\T^2)$ as long as $s \geq s_{\text{crit}}$ (with a strict inequality when $p=3$), where $s_{\text{crit}}$ is the critical regularity as defined in \eqref{scrit1}. For details, see Theorem \ref{THM:LWP3} and Remark \ref{REM:LWP3}.
\end{remark}

\section{Global well-posedness of SvNLW}
\label{SEC:GWP1}

In this section, we aim to prove Theorem \ref{THM:GWP1}, i.e. pathwise global well-posedness of SvNLW \eqref{SvNLW}. As mentioned in Subsection \ref{SUBSEC:Sv}, we prove Theorem \ref{THM:GWP1} by studying the equation \eqref{SvNLW_res} for $v$ with $(v, \dt v)|_{t=0} = (u_0, u_1)$, for given initial data $(u_0, u_1) \in \H^s (\T^2)$ of \eqref{SvNLW}.

Fix an arbitrary $T \geq 1$. In view of Proposition \ref{PROP:LWP1}, in order to show well-posedness of \eqref{vNLW} on $[0,T]$, it suffices to show that the $\H^\s$-norm of the solution $\vec v(t) = (v(t), \dt v(t))$ to \eqref{vNLW} remains finite on $[0,T]$, where $\s$ is as defined in Proposition \ref{PROP:LWP1}. This will allow us to iteratively apply the pathwise local well-posedness result in Proposition \ref{PROP:LWP1}.

In fact, we show that the solution $\vec v(t)$ belongs to $\H^1(\T^2)$. Let $0 < t \leq 1$. From Lemma \ref{LEM:Sch}, we have
\begin{equation}
\| V(t)(u_0, u_1) \|_{\H^1} \les ( 1 + t^{-1+\s} ) \| (u_0,u_1) \|_{\H^\s}.
\label{H1bd}
\end{equation}

\noi
Then, let $0 < T_0 \leq  1$ be the local existence time as in the proof of Proposition \ref{PROP:LWP1}. Thus, given $s \geq \s,$ by \eqref{L1}, \eqref{H1bd}, Lemma \ref{LEM:lin2}, H\"older's inequality, and Sobolev's inequality, we have that for $0 < t \leq T_0$,
\begin{equation}
\begin{split}
\| \vec v (t) \|_{\H^1} &\les (1+t^{-1+\s}) \| (u_0,u_1) \|_{\H^\s} + \| (v + \Psi)^p \|_{L_{T_0}^1 L_x^2} \\
&\les (1+t^{-1+\s}) \| (u_0,u_1) \|_{\H^s} + T_0^\ta \Big( \| v \|_{L_{T_0}^q L_x^r}^p + \| \Psi \|_{C_{T_0} W_x^{\frac 12 - \al - \eps, \infty}}^p \Big),
\end{split}
\label{vH1}
\end{equation}
where $\dl > 0$, $\eps > 0$ are sufficiently small, $\ta > 0$, and $q, r$ are as defined in the statement of Proposition \ref{PROP:LWP1}. Here, due to Lemma \ref{LEM:regPsi}, we can fix a good $\o \in \O$ such that $\Psi = \Psi(\o) \in C([0,T_0]; W^{\frac 12 - \al - \eps, \infty}(\T^2))$ for $\al < \frac 12$ and sufficiently small $\eps > 0$, so that we know from \eqref{vH1} that $\| \vec v (t) \|_{\H^1} < \infty$. A standard argument then shows that $\vec v \in C((0, T_0]; \H^1(\T^2))$. Thus, our main goal is to control the $\H^1$-norm of $\vec v (t)$ on $[0,T]$ by bounding the energy $E(\vec v)$ defined in \eqref{defE}.

For the following computation, we need to work with the smooth solution $(v_N, \dt v_N)$ to the truncated equation with initial data $(\pi_N v_0, \pi_N v_1)$, where $\pi_N$ is the frequency truncation onto the frequencies $\{ |n| \leq N \}$. After establishing an upper bound for $E(\vec v(t))$ with the implicit constant independent of $N$, we can take $N \to \infty$ by using Proposition \ref{PROP:LWP1} (specifically, the continuous dependence of a solution on the initial data). Here, we omit details and work with $(v, \dt v)$ instead for simplicity. See, for example, \cite{OP16} for a standard argument.

\subsection{Case $1 < p \leq 3$}
\label{SUBSEC:GWP1}

In this case, we follow the globalization argument by Burq-Tzvetkov \cite{BT14}. For simplicity of notation, we set $E(t) = E(\vec v(t))$.

Given $T > 0$, we fix $0 < t \leq T$. By \eqref{defE}  and \eqref{SvNLW_res}, we have
\begin{equation}
\begin{split}
\dt E(t) &= \int_{\T^2} \dt v \big( \dt^2 v + (1-\Dl) v + |v|^{p-1}v \big) dx \\
&\leq -\int_{\T^2} \dt v\big( |v + \Psi|^{p-1}(v + \Psi) - |v|^{p-1}v \big) dx.
\end{split}
\label{E1}
\end{equation}

\noi
Let $F(u) = |u|^{p-1}u$, so that we can compute $F'(u) = p|u|^{p-1}$. Thus, by the fundamental theorem of calculus, we have
\begin{equation}
\begin{split}
|v + \Psi|^{p-1}(v + \Psi) - |v|^{p-1}v &= F(v+\Psi) - F(v) \\
&= \Psi \int_0^1 F'(v + \tau \Psi) d\tau \\
&\les |\Psi| |v|^{p-1} + |\Psi|^p.
\end{split}
\label{E2}
\end{equation}

\noi
Combining \eqref{E1} and \eqref{E2} and then applying the Cauchy-Schwartz inequality, we obtain
\begin{equation}
\begin{split}
\dt E(t) &\les \| \Psi \|_{L_x^\infty} \int_{\T^2} |\dt v| |v|^{p-1} dx + \int_{\T^2} |\dt v| |\Psi|^p dx \\
&\leq \| \Psi \|_{L_x^\infty} \bigg( \int_{\T^2} (\dt v)^2 dx \bigg)^{\frac 12} \bigg( \int_{\T^2} |v|^{2(p-1)} dx \bigg)^{\frac 12} + \| \Psi \|_{L_x^\infty}^{2p} \bigg( \int_{\T^2} (\dt v)^2 dx \bigg)^{\frac 12} \\
&\leq C(\Psi) E(t),
\end{split}
\label{E3}
\end{equation}

\noi
as long as $2(p-1) \leq p+1$, or equivalently, $p \leq 3$. By Gronwall's inequality on \eqref{E3}, we get
\[ E(t) \les e^{C(\Psi) t} \]
for any $0 < t \leq T$.

\subsection{Case $3 < p \leq 5$}
\label{SUBSEC:GWP2}

In this case, we follow the idea introduced by Oh-Pocovnicu \cite{OP16}. See also \cite{SX, LM2, MPTW} for similar arguments. In this setting, we let $\al < \frac{2}{p-1} - \frac 12$, the reason of which will become clear in the following steps.

By \eqref{defE}, \eqref{SvNLW_res}, and Taylor's theorem, we have
\begin{equation}
\begin{split}
\dt E(t) &= \int_{\T^2} \dt v \big( \dt^2 v + (1-\Dl) v + |v|^{p-1}v \big) dx \\
&= -\int_{\T^2} \dt v \big( |v + \Psi|^{p-1} (v + \Psi) - |v|^{p-1} v \big) dx - \int_{\T^2} \big( D^{\frac 12} \dt v \big)^2 dx \\
&\leq -p \int_{\T^2} \dt v \cdot |v|^{p-1} \Psi dx - \frac{p(p-1)}{2} \int_{\T^2} \dt v \cdot |v + \ta \Psi|^{p-3} (v + \ta \Psi) \Psi^2 dx \\
&=: A_1 + A_2,
\end{split}
\label{A1A2}
\end{equation}

\noi
where $\ta \in (0,1)$. To estimate $A_2$, by the Cauchy-Schwartz inequality and  Cauchy's inequality, we have
\begin{equation}
\begin{split}
|A_2| &\les \int_{\T^2} |\dt v| \big( |v|^{p-2} \Psi^2 + \Psi^p \big) dx \\
&\les \bigg( \int_{\T^2} (\dt v)^2 dx \bigg)^{1/2} \bigg( \| \Psi \|_{L_x^{\infty}}^4 \int_{\T^2} |v|^{2(p-2)} dx + \| \Psi \|_{L_x^{2p}}^{2p} \bigg)^{1/2} \\
&\les (1 + \| \Psi \|_{L_x^{\infty}}^4) E(t) + \| \Psi \|_{L_x^{2p}}^p,
\end{split}
\label{A2}
\end{equation}

\noi
where in the last inequality, we need $2(p-2) \leq p+1$, which is equivalent to $p \leq 5$. To estimate $A_1$, for $0 < t_1 \leq t_2 \leq T$, by integration by parts and Young's inequality, we have
\begin{equation}
\begin{split}
\int_{t_1}^{t_2} A_1 dt' &= - \int_{t_1}^{t_2} \int_{\T^2} \dt (|v|^{p-1}v) \Psi dx dt' \\
&= - \int_{\T^2} |v(t_2)|^{p-1} v(t_2) \Psi(t_2) dx + \int_{\T^2} |v(t_1)|^{p-1} v(t_1) \Psi(t_1) dx \\
&\hphantom{X}\, + \int_{t_1}^{t_2} \int_{\T^2} |v|^{p-1} v (\dt\Psi) dx dt' \\
&\les \eps \| v(t_2) \|_{L_x^{p+1}}^{p+1} + \frac{1}{\eps} \| \Psi(t_2) \|_{L_x^{p+1}}^{p+1} + \| v(t_1) \|_{L_x^{p+1}}^{p+1} + \| \Psi(t_1) \|_{L_x^{p+1}}^{p+1} \\
&\hphantom{X}\, + \int_{t_1}^{t_2} \int_{\T^2} |v|^{p-1} v (\dt\Psi) dx dt',
\end{split}
\label{A1a}
\end{equation}

\noi
where $0 < \eps < 1$. We see in Remark \ref{REM:dPsi} that $\dt \Psi \in C([0,T]; W^{-\frac 12 - \al -, \infty} (\T^2))$. By duality, H\"older's inequality, Lemma \ref{LEM:chain} (i), and Lemma \ref{LEM:Gag}, we obtain
\begin{equation}
\begin{split}
\int_{t_1}^{t_2} &\int_{\T^2} |v|^{p-1} v (\dt\Psi) dx dt' \\
&= \int_{t_1}^{t_2} \int_{\T^2} \jb{\nb}^{\frac 12+\al+} (|v|^{p-1}v) \jb{\nb}^{-\frac 12-\al -} (\dt \Psi) dx dt' \\
&\les \int_{t_1}^{t_2} \| v \|_{L_x^{p+1}}^{p-1} \big\| \jb{\nb}^{\frac 12+\al +} v(t')\big\|_{L_x^{\frac{p+1}{2}}} \big\| \jb{\nb}^{-\frac 12-\al -} (\dt\Psi)(t') \big\|_{L_x^\infty} dt' \\
&\les \| \dt\Psi \|_{C_TW_x^{-\frac 12 - \al -, \infty}} \int_{t_1}^{t_2} E(t')^{\frac{p-1}{p+1}} \| \jb{\nb} v \|_{L_x^2}^{\frac{2}{p-1}} \| v \|_{L_x^{p+1}}^{\frac{p-3}{p-1}} dt' \\
&\les \| \dt\Psi \|_{C_TW_x^{-\frac 12 - \al -, \infty}} \int_{t_1}^{t_2} E(t') dt',
\end{split}
\label{A1b}
\end{equation}

\noi
where we require that
\[ \frac 12 + \al + = \frac{2}{p-1},\]
which is equivalent to $\al < \frac{2}{p-1} - \frac 12$.
By combining \eqref{A1A2}, \eqref{A2}, \eqref{A1a}, and \eqref{A1b}, we have
\[E(t_2) \leq (1+ C_1(\Psi)) \int_{t_1}^{t_2} E(t') dt' + C_2(\Psi, v(t_1)). \]
By Gronwall's inequality, we get
\[E(t) \les e^{C(\Psi) t}\]
for any $0 < t \leq T$.

\subsection{Case $p > 5$}
\label{SUBSEC:GWP3}

In this case, we follow the idea by Latocca \cite{Lat}. In this setting, we also let $\al < \frac{2}{p-1} - \frac 12$.

We need  the following lemma to close the energy estimates in the Gronwall argument. We define $\b_p := \lceil \frac{p-3}{2} \rceil$, $F(u) := |u|^{p-1}u$, and $s_p := \frac{p-3}{p-1}$.

\begin{lemma}\label{LEM:Bes_bd}
For any $0 < t \leq T$ and every integer $1 \leq k \leq \b_p$, we have
\begin{align*}
\bigg| \int_{\T^2} F^{(k-1)}(v(t)) \Psi(t)^{k-1} &\dt\Psi(t) dx \bigg| \\
&\les g\big( \| \Psi \|_{L^\infty([0,T]; X)}, \| \jb{\nb}^{-1} \dt\Psi \|_{L^\infty([0,T]; Y)} \big) (1+E(t)),
\end{align*}

\noi
where $g$ is a polynomial with positive coefficients, and
\[X := L^\infty (\T^2) \cap B_{\frac{p+1}{2},1}^{1-s_p} (\T^2) \quad \text{and} \quad Y := L^\infty (\T^2) \cap B_{\infty,1}^{s_p} (\T^2).\]
\end{lemma}

Note that given $\al < \frac{2}{p-1} - \frac 12$, by Lemma \ref{LEM:regPsi}, Remark \ref{REM:dPsi}, and Lemma \ref{LEM:Bes} (ii), we have
\[ g\big( \| \Psi \|_{L^\infty([0,T]; X)}, \| \jb{\nb}^{-1} \dt\Psi \|_{L^\infty([0,T]; Y)} \big) < \infty \]
almost surely.

\medskip
Let us first assume Lemma \ref{LEM:Bes_bd} and work on the energy bound. As in the case when $p > 3$, we can compute that for $0 < t \leq T$,
\begin{equation}
\dt E(t) \leq -\int_{\T^2} \dt v (F(v + \Psi) - F(v)) dx.
\label{dtE}
\end{equation}

\noi
For our convenience we compute that for $k \in \Z_+$,
\[ F^{(k)} (u) = \begin{cases}
C_{p,k} |u|^{p-k-1}u & \text{for $k$ even}, \\
C_{p,k} |u|^{p-k} & \text{for $k$ odd}.
\end{cases} \]
By Taylor's formula at the point $v(t,x)$ with integral remainder up to the order $\b_p = \lceil \frac{p-3}{2} \rceil$, we have
\[ F(v + \Psi) - F(v) = \sum_{k=1}^{\b_p} \frac{1}{k!} F^{(k)}(v) \Psi^k + \int_v^{v+\Psi} \frac{F^{(\b_p + 1)}(\tau)}{\b_p!} (v+\Psi-\tau)^{\b_p} d\tau. \]
Let $0 < t_1 \leq t_2 \leq T$. By integrating \eqref{dtE} from $t_1$ to $t_2$, we can write
\begin{equation}
E(t_2) \leq E(t_1) + \sum_{k=1}^{\b_p} C_k I_k + C_p R,
\label{Et}
\end{equation}

\noi
where
\[ I_k := -\int_{t_1}^{t_2} \int_{\T^2} \dt v F^{(k)}(v) \Psi^k dx dt'  \quad \text{for} \quad 1 \leq k \leq \b_p, \]
\[ R := -\int_{t_1}^{t_2} \int_{\T^2} \int_v^{v+\Psi} \dt v F^{(\b_p+1)}(\tau) (v+\Psi-\tau)^{\b_p} d\tau dx dt'. \]

We first estimate $R$. Note that for $\tau \in [v, v+\Psi]$, we have
\[ |F^{(\b_p+1)}(\tau)| \les |v|^{p-\b_p-1} + |\Psi|^{p-\b_p-1}. \]
Thus, by H\"older's inequality and Young's inequality, we have
\begin{equation}
\begin{split}
R &\les \int_{t_1}^{t_2} \int_{\T^2} \dt v (|v|^{p-\b_p-1} |\Psi|^{\b_p+1} + |\Psi|^p) dx dt' \\
&\les \int_{t_1}^{t_2} \| \dt v(t') \|_{L_x^2} \| v(t') \|_{L_x^{p+1}}^{p-\b_p-1} \| \Psi(t') \|_{L_x^{r_p(\b_p+1)}}^{\b_p+1} dt' + \int_{t_1}^{t_2} \| \dt v(t') \|_{L_x^2}^2 dt' \\
&\hphantom{X}\, + \| \Psi \|_{L_T^{2p} L_x^{2p}}^{2p} \\
&\leq \| \Psi \|_{L_T^{2p} L_x^{2p}}^{2p} + \bigg( 1 + \| \Psi \|_{L_T^\infty L_x^{r_p(\b_p+1)}}^{\b_p+1} \bigg) \int_{t_1}^{t_2} \max\bigg\{ E(t'), E(t')^{\frac 12 + \frac{p-\b_p-1}{p+1}} \bigg\} dt',
\end{split}
\label{defR}
\end{equation}

\noi
where $r_p$ satisfies $\tfrac 12 + \tfrac{p-\b_p-1}{p+1} + \frac{1}{r_p} = 1$. Since $\b_p = \lceil \frac{p-3}{2} \rceil \geq \frac{p-3}{2}$, we have $\frac{p-\b_p-1}{p+1} \leq \frac 12$, so that
\[ R \les \| \Psi \|_{L_T^{2p} L_x^{2p}}^{2p} + \bigg( 1 + \| \Psi \|_{L_T^\infty L_x^{r_p(\b_p+1)}}^{\b_p+1} \bigg) \int_{t_1}^{t_2} (1+E(t')) dt'. \]

We now estimate $I_k$. By Fubini's theorem and integration by parts in time, we have
\begin{equation}
\begin{split}
|I_k| &= \bigg| -\int_{\T^2} \int_{t_1}^{t_2} \dt (F^{(k-1)}(v)) \Psi^k dt' dx \bigg| \\
&\leq \bigg| \int_{\T^2} F^{(k-1)}(v(t_2)) \Psi^k(t_2) dx \bigg| + \bigg| \int_{\T^2} F^{(k-1)}(v(t_1)) \Psi^k(t_1) dx \bigg| \\
&\hphantom{X}\, + \bigg| k \int_{\T^2} \int_{t_1}^{t_2} F^{(k-1)}(v(t')) \Psi(t')^{k-1} \dt\Psi(t') dt' dx \bigg| \\
&\les \int_{\T^2} |v(t_2)|^{p-k+1} |\Psi(t_2)|^k + |v(t_1)|^{p-k+1} |\Psi(t_1)|^k dx \\
&\hphantom{X}\, + \bigg| \int_{t_1}^{t_2} \int_{\T^2} F^{(k-1)}(v(t')) \Psi(t')^{k-1} \dt\Psi(t') dx dt' \bigg| \\
&= J_k + K_k.
\end{split}
\label{Ik}
\end{equation}

\noi
To handle $J_k$, by H\"older's inequality and Young's inequality, we obtain
\begin{equation}
\begin{split}
J_k &\leq E(t_2)^{\frac{p-k+1}{p+1}} \| \Psi(t_2) \|_{L_x^{p+1}}^k + E(t_1)^{\frac{p-k+1}{p+1}} \| \Psi(t_1) \|_{L_x^{p+1}}^k \\
&\leq \eps E(t_2) + C_1 E(t_1) + C_2 \| \Psi \|_{L_T^\infty L_x^{p+1}}^{p+1},
\end{split}
\label{Jk}
\end{equation}

\noi
where $0 < \eps < 1$. To deal with $K_k$, by Lemma \ref{LEM:Bes_bd}, 
\begin{equation}
K_k \les g\big( \| \Psi \|_{L^\infty([0,T]; X)}, \| \jb{\nb}^{-1} \dt\Psi \|_{L^\infty([0,T]; Y)} \big) \bigg( 1+\int_{t_1}^{t_2} E(t') dt' \bigg). 
\label{Kk}
\end{equation}

\noi
By combining \eqref{Et}, \eqref{defR}, \eqref{Ik}, \eqref{Jk}, \eqref{Kk}, we obtain
\begin{align*}
E(t_2) &\les \bigg( 1+\| \Psi \|_{L_T^\infty L_x^{r_p(\b_p+1)}}^{\b_p+1} + g\big( \| \Psi \|_{L^\infty([0,T]; X)}, \| \jb{\nb}^{-1} \dt\Psi \|_{L^\infty([0,T]; Y)} \big) \bigg) \\
&\hphantom{X}\, \times \bigg( 1+\int_{t_1}^{t_2} E(t') dt' \bigg) + \| \Psi \|_{L_T^{2p} L_x^{2p}}^{2p} + \| \Psi \|_{L_T^\infty L_x^{p+1}}^{p+1} + E(t_1).
\end{align*}

\noi
We can then use Gronwall's inequality to get the desired bound.

\medskip
We now provide the proof of Lemma \ref{LEM:Bes_bd}.

\begin{proof}[Proof of Lemma \ref{LEM:Bes_bd}]
Recall that $s_p = \frac{p-3}{p-1}$. We first consider the case when $k \geq 2$. By the Fourier-Plancherel theorem, we have
\begin{equation}
\begin{split}
\bigg| \int_{\T^2} F^{(k-1)}(v(t)) &\Psi(t)^{k-1} \dt\Psi(t) dx \bigg| \\
&= \bigg| \sum_{j'=-1}^1 \sum_{j \geq 0} \int_{\T^2} \P_j (F^{(k-1)}(v(t)) \Psi(t)^{k-1}) \P_{j+j'} (\dt\Psi(t)) dx \bigg| \\
&\les \sum_{j > 2} \int_{\T^2} |\P_j (F^{(k-1)}(v(t)) \Psi(t)^{k-1})| |\P_j (\dt\Psi(t))| dx \\
&\hphantom{X} + \sum_{j = 0}^2 \int_{\T^2} |\P_j (F^{(k-1)}(v(t)) \Psi(t)^{k-1})| |\P_j (\dt\Psi(t))| dx \\
&=: I_1 + I_2.
\end{split}
\label{I1I2}
\end{equation}

\noi
Let $r_k := \frac{(k-1)(p+1)}{k}$. To estimate $I_2$, by H\"older's inequality, Bernstein's inequality, and Young's inequality, we have
\begin{equation*}
\begin{split}
I_2 &\les \| \Psi(t) \|_{L_x^{r_k}}^{k-1} \| v(t) \|_{L_x^{p+1}}^{p-k+1} \sum_{j=0}^2 \| \P_j \dt\Psi(t) \|_{L_x^\infty} \\
&\les \| \Psi(t) \|_{L_x^{r_k}}^{k-1} \| \jb{\nb}^{-1} \dt\Psi(t) \|_{L_x^\infty} E(t)^{\frac{p-k+1}{p+1}} \\
&\les E(t) + \| \Psi \|_{L_T^\infty L_x^{r_k}}^{r_k} \| \jb{\nb}^{-1} \dt\Psi \|_{L_T^\infty  L_x^\infty}^{\frac{p+1}{k}}.
\end{split}
\end{equation*}

\noi
It remains to estimate $I_1$. By H\"older's inequality, Bernstein's inequality, and then H\"older's inequality for series,
\begin{align*}
I_1 &\les \sum_{j > 2} 2^{j(1-s_p)} \| \P_j (F^{(k-1)}(v(t)) \Psi(t)^{k-1}) \|_{L_x^1} 2^{j s_p} \| \P_j (\jb{\nb}^{-1}\dt\Psi(t)) \|_{L_x^\infty} \\
&\leq \| F^{(k-1)}(v(t)) \Psi(t)^{k-1} \|_{B_{1,\infty}^{1-s_p}} \| \jb{\nb}^{-1}\dt\Psi \|_{B_{\infty,1}^{s_p}}.
\end{align*}

\noi
Then, by Corollary \ref{COR:bilin}, we have
\begin{equation}
\begin{split}
\| F^{(k-1)}(v(t)) &\Psi(t)^{k-1} \|_{B_{1,\infty}^{1-s_p}} \les \|  F^{(k-1)} (v(t)) \|_{B_{\frac{p+1}{p+2-k}, \infty}^{1-s_p}} \| \Psi(t)^{k-1} \|_{L^{\frac{p+1}{k-1}}} \\
&\hphantom{X}\, + \| |v(t)|^{p-k+1} \|_{L^{\frac{p+1}{p-k+1}}} \| \Psi(t)^{k-1} \|_{B_{p_k,\infty}^{1-s_p}} \\
&\les \| F^{(k-1)} (v(t)) \|_{B_{\frac{p+1}{p+2-k}, \infty}^{1-s_p}} \| \Psi(t) \|_{L^{p+1}}^{k-1} + E(t)^{\frac{p-k+1}{p+1}} \| \Psi(t)^{k-1} \|_{B_{p_k,\infty}^{1-s_p}},
\end{split}
\label{I1-1}
\end{equation}

\noi
where $p_k$ satisfies $\frac{1}{p_k} + \frac{p-k+1}{p+1} = 1$. By Lemma \ref{LEM:chain} (ii), we have
\begin{equation}
\begin{split}
\| \Psi(t)^{k-1} \|_{B_{p_k,\infty}^{1-s_p}} &\les \| \Psi(t) \|_{B_{p_k,\infty}^{1-s_p}} \| \Psi(t) \|_{L^\infty}^{k-2} \\
&\leq \| \Psi(t) \|_{B_{\frac{p+1}{2},\infty}^{1-s_p}} \| \Psi(t) \|_{L^\infty}^{k-2}.
\end{split}
\label{I1-2}
\end{equation} 

\noi
By Lemma \ref{LEM:chain} (ii), Lemma \ref{LEM:Bes} (ii), and Lemma \ref{LEM:Gag}, we have
\begin{align*}
\|  F^{(k-1)} (v(t)) \|_{B_{\frac{p+1}{p+2-k}, \infty}^{1-s_p}} &\les \| v(t) \|_{B_{\frac{p+1}{2}, \infty}^{1-s_p}} \big\| |v(t)|^{p-k} \big\|_{L^{\frac{p+1}{p-k}}} \\
&\les \| v(t) \|_{W^{1-s_p, \frac{p+1}{2}}} E(t)^{\frac{p-k}{p+1}} \\
&\les \| \jb{\nb} v(t) \|_{L^2}^{1-\b} \| v(t) \|_{L^{p+1}}^\b E(t)^{\frac{p-k}{p+1}},
\end{align*}

\noi
where $\b \in [0,s_p]$ satisfies $\frac{2}{p+1} = \frac{1-s_p}{2} + \frac{\b}{p+1}$, and so $\b = \frac{p-3}{p-1} = s_p$. Thus, we obtain

\begin{equation}
\| F^{(k-1)} (v(t)) \|_{B_{\frac{p+1}{p+2-k}, \infty}^{1-s_p}} \les E(t)^{\frac{1-\b}{2} + \frac{\b}{p+1} + \frac{p-k}{p+1}} = E(t)^{\frac{2}{p+1} + \frac{p-k}{p+1}} \les 1 + E(t).
\label{I1-3}
\end{equation}

\noi
By combining \eqref{I1-1}, \eqref{I1-2}, and \eqref{I1-3}, we obtain the desired bound for $I_1$.

For the case when $k=1$, after \eqref{I1I2}, we have the estimate $I_2 \les E(t) + \| \jb{\nb}^{-1} \dt\Psi \|_{L_T^\infty  L_x^\infty}^{p+1}$. For the term $I_1$, by the estimate in \eqref{I1-3}, we have
\[ I_1 \les \| F(v(t)) \|_{B_{1,\infty}^{1-s_p}} \| \jb{\nb}^{-1} \dt\Psi \|_{B_{\infty,1}^{s_p}} \les \| \jb{\nb}^{-1} \dt\Psi \|_{B_{\infty,1}^{s_p}} (1+E(t)), \]
as desired.
\end{proof}

\appendix

\section{On local well-posedness of subcritical vNLW}
\label{SEC:LWP3}

In this appendix, we aim to show that the deterministic viscous NLW is locally well-posed in $\H^s(\T^2)$ with $s \geq s_{\text{crit}}$, where we recall that $s_{\text{crit}}$ is defined by
\begin{equation}
s_{\text{crit}} := \max\bigg( 1-\frac{2}{p-1}, 0 \bigg).
\label{scrit}
\end{equation}

\noi
More precisely, we prove local well-posedness of the following subcritical vNLW:
\begin{equation}
\begin{cases}
\dt^2 u + (1-\Dl) u + D \dt u \pm |u|^{p-1}u = 0 \\
(u, \dt u)|_{t=0} = (u_0, u_1),
\end{cases}
\label{vNLWcrit}
\end{equation}

\noi
where $(u_0, u_1) \in \H^s(\T^2)$ and $s \geq s_{\text{crit}}$ (with a strict inequality when $p = 3$). To achieve this, we will need the inhomogeneous Strichartz estimates for the linear viscous wave equation on $\T^2$.

\subsection{The inhomogeneous Strichartz estimates}

In this subsection, we prove the Strichartz estimates for the inhomogeneous linear viscous wave equation on $\T^d$. To achieve this, we first establish the following estimate for the linear operator $S(t)$ defined in \eqref{defS}.

\begin{lemma}\label{LEM:Sest}
Let $1 \leq p \leq 2 \leq q \leq \infty$. Then, we have
\[ \| S(t) \phi \|_{L^q(\T^d)} \les t^{1-d(\frac{1}{p} - \frac{1}{q})} \| \phi \|_{L^p(\T^d)} \]
for any $0 < t \leq 1$
\end{lemma}

\begin{proof}
By \eqref{defS} and applying the Schauder estimate (Lemma \ref{LEM:Sch}) twice, we obtain
\begin{align*}
\| S(t) \phi \|_{L^q(\T^d)} &= \bigg\| e^{-\frac{D}{4}t} \frac{\sin(t\jbb{D})}{\jbb{D}} e^{-\frac{D}{4}t} \phi \bigg\|_{L^q(\T^d)} \\
&\les t^{-d(\frac 12 - \frac{1}{q})} \bigg\| \frac{\sin(t\jbb{D})}{\jbb{D}} e^{-\frac{D}{4}t} \phi \bigg\|_{L^2(\T^d)} \\
&\leq t^{1-d(\frac 12 - \frac{1}{q})} \big\| e^{-\frac{D}{4}t} \phi \big\|_{L^2(\T^d)} \\
&\les t^{1-d(\frac{1}{p} - \frac{1}{q})} \| \phi \|_{L^p(\T^d)},
\end{align*}
as desired.
\end{proof}

We now establish the Strichartz estimates for the inhomogeneous linear viscous wave equation on $\T^d$. We say that $u$ is a solution to the following inhomogeneous linear viscous wave equation:
\begin{equation}
\begin{cases}
\dt^2 u + (1 - \Dl) u + D \dt u = f \\
(u, \dt u) |_{t=0} = (\phi_0, \phi_1),
\end{cases}
\label{ivLW}
\end{equation}

\noi
if $u$ satisfies the following Duhamel formulation:
\[ u(t) = V(t)(\phi_0, \phi_1) + \int_0^t S(t-t') f(t') dt', \]
where $V(t)$ and $S(t)$ are as defined in \eqref{defV} and \eqref{defS}, respectively.

\begin{lemma}\label{LEM:str}
Given $s \geq 0$, suppose that $1 < \wt q \leq 2 < q < \infty$, $1 \leq \wt r \leq 2 \leq r \leq \infty$ satisfy the following scaling condition:
\begin{equation}
\frac{1}{q} + \frac{d}{r} = \frac{d}{2}-s = \frac{1}{\wt q} + \frac{d}{\wt r} - 2.
\label{str_cond}
\end{equation}
Then, a solution $u$ to the inhomogeneous linear viscous wave equation \eqref{ivLW} satisfies the following inequality:
\begin{equation}
\| (u, \dt u) \|_{C_T \H^s_x (\T^d)} + \| u \|_{L_T^q L_x^r (\T^d)} \les \| (\phi_0, \phi_1) \|_{\H^s (\T^d)} + \| f \|_{L_T^{\wt q} L_x^{\wt r} (\T^d)},
\label{str}
\end{equation}
for all $0 < T \leq 1$.
\end{lemma}

\begin{proof}
By \eqref{defV}, we have
\begin{equation}
\big\| \big(V(t)(\phi_0,\phi_1), \dt V(t)(\phi_0,\phi_1)\big) \big\|_{C_T \H^s_x (\T^d)} \les \| (\phi_0, \phi_1) \|_{\H^s (\T^d)}.
\label{str1}
\end{equation}

\noi
By Lemma \ref{LEM:hom_str}, we have
\begin{equation}
\| V(t)(\phi_0, \phi_1) \|_{L_T^q L_x^r (\T^d)} \les \| (\phi_0, \phi_1) \|_{\H^s (\T^d)}.
\label{str2}
\end{equation}

\noi
We then use Lemma 3.5 in \cite{KC1} (which is in the $\R^d$ setting, but the proof also works in the $\T^d$ setting with Lemma \ref{LEM:Sest} in hand) to obtain
\begin{equation}
\bigg\| \int_0^t S(t-t') f(t') dt' \bigg\|_{L_T^q L_x^r (\T^d)} \les \| f \|_{L_T^{\wt q} L_x^{\wt r} (\T^d)}.
\label{str3}
\end{equation}

\noi
It remains to show
\begin{equation}
\bigg\| \int_0^t S(t-t') f(t') dt' \bigg\|_{C_T H_x^s (\T^d)} \les \| f \|_{L_T^{\wt q} L_x^{\wt r} (\T^d)}
\label{str4}
\end{equation}
and
\begin{equation}
\bigg\| \dt \int_0^t S(t-t') f(t') dt' \bigg\|_{C_T H_x^{s-1} (\T^d)} \les \| f \|_{L_T^{\wt q} L_x^{\wt r} (\T^d)},
\label{str5}
\end{equation}
so that \eqref{str} follows from \eqref{str1}, \eqref{str2}, \eqref{str3}, \eqref{str4}, and \eqref{str5}.

To show that the inequality \eqref{str4} holds, we use the Littlewood-Paley decomposition as in Lemma 3.6 in \cite{KC1}. In view of the proof of Lemma 3.6 in \cite{KC1}, we know that it suffices to show \eqref{str4} for all $f$ such that $\ft f$ is supported in $\{ n \in \Z^d: 2^{j-1} \leq |n| \leq 2^{j+1} \}$ for all $j \in \Z_+$ (the case for $\{ n \in \Z^d: 0 \leq |n| \leq 2\}$ follows in a similar manner) with the underlying constant independent of $j$. Fix $0 < t < T$. By Minkowski's integral inequality, H\"older's inequality in $n$, Hausdorff-Young inequality, H\"older's inequality in $t'$ (along with the fact that the number of lattice points inside a ball of radius $R$ in $\R^d$ is $O(R^d)$), and a change of variable, we have
\begin{align*}
\bigg\| \int_0^t S(t-t') &f(t') dt' \bigg\|_{C_T H_x^s} \\
&\les \int_0^t \bigg( \sum_{n \in \Z^2 } |n|^{2s} \bigg| e^{-\frac{|n|}{2} (t-t')} \frac{\sin((t-t') \jbb{n})}{\jbb{n}} \ft f (t', n) \bigg|^2 \bigg)^{1/2} dt' \\
&\les 2^{(j+1)s} \int_0^t (t-t') e^{2^{j-2}(t-t')} \bigg( \sum_{n \in \Z^2} \big|\ft f (t', n)\big|^2 \bigg)^{1/2} dt' \\
&\les 2^{(j+1)s} \int_0^t (t-t') e^{2^{j-2}(t-t')} \big( 2^{(j+1)d} \big)^{\frac{\wt r' - 2}{2\wt r '}} \big\| \ft f (t', n) \big\|_{\ell_n^{\wt r '}} dt' \\
&\les 2^{(j+1) (s + \frac{d}{\wt r} - \frac{d}{2})} \bigg( \int_0^t \Big| (t-t') e^{2^{j-2}(t-t')} \Big|^{\wt q '} dt' \bigg)^{1/\wt q '} \| f \|_{L_T^{\wt q} L_x^{\wt r}} \\
&\les 2^{(j+1) (s + \frac{d}{\wt r} - \frac{d}{2} + \frac{1}{\wt q} - 2)} \| f \|_{L_T^{\wt q} L_x^{\wt r}}.
\end{align*}

\noi
By using the second equality in the scaling condition \eqref{str_cond}, we obtain the desired inequality with the underlying constant independent of $j$, and so the inequality \eqref{str4} follows. The inequality \eqref{str5} follows in a similar manner.
\end{proof}

\begin{remark}\rm
As in the case of the homogeneous Strichartz estimates (Lemma \ref{LEM:hom_str}), the Strichartz estimates for the inhomogeneous linear viscous wave equation on $\T^d$ also hold for a larger class of pairs $(q,r)$ and $(\wt q, \wt r)$ compared to the Strichartz estimates for the usual linear wave equations \cite{GV, LS, KT, GKO}. Again, this is due to the parabolic smoothing effect. Note that this is also true on $\R^d$ (see \cite{KC1}).
\end{remark}

We complete this subsection by making the following observation. Recall that we are considering the viscous NLW on $\T^2$ with nonlinearity $|u|^{p-1}u$ for $p>1$. Suppose that we can find pairs $(q,r)$ and $(\wt q, \wt r)$ satisfying the scaling condition $\eqref{str_cond}$ such that
\[ q > p \wt q \qquad \text{and} \qquad r \geq p \wt r. \] 
Then, by H\"older's inequality and the fact that $|\T^2| = 1$, we have
\[ \big\| |u|^{p-1}u \big\|_{L_T^{\wt q} L_x^{\wt r}} \leq T^{\frac{1}{\wt q} - \frac{p}{q}} \| u \|^p_{L_T^q L_x^r}. \]
Note that the power of $T$ is positive when $q > p \wt q$. The following lemma shows that there exist such pairs $(q,r)$ and $(\wt q, \wt r)$.

\begin{lemma}\label{LEM:qr}
Let $s_{\text{crit}}$ be as defined in \eqref{scrit}. Given $s_{\text{crit}} < s < 1$, there exist $1 < \wt q \leq 2 < q < \infty$, $1 \leq \wt r \leq 2 \leq r \leq \infty$ satisfying the scaling condition \eqref{str_cond} such that
\begin{equation}
q > p \wt q \qquad \text{and} \qquad r \geq p \wt r.
\label{qr_cond}
\end{equation}
\end{lemma}
\begin{proof}
In view of Lemma 3.3 in \cite{GKO}, given $0 < s < 1$, we have
\[ \min \bigg( \frac{q}{\wt q}, \frac{r}{\wt r} \bigg) \leq \frac{3-s}{1-s}, \]
and the equality holds by taking, for example, 
\begin{equation}
(q, r) = \bigg( \frac{3-s}{1-s}\dl, \frac{2}{1-s-\frac{1-s}{(3-s)\dl}} \bigg) \qquad \text{and} \qquad (\wt q, \wt r) = \bigg( \dl, \frac{2}{3-s-\frac{1}{\dl}} \bigg),
\label{qr}
\end{equation}
where $\dl = \dl(s) > 1$ is sufficiently close to 1. Moreover, we note that $\frac{3-s}{1-s} > p$ if and only if $s > 1 - \frac{2}{p-1}$. Thus, as long as $s_{\text{crit}} < s < 1$, there exist pairs $(q,r)$ and $(\wt q, \wt r)$ that satisfy $\eqref{qr_cond}$.
\end{proof}

\begin{remark}\label{REM:qr}\rm
In the case when $p > 3$ and $s = s_{\text{crit}} = 1 - \frac{2}{p-1} > 0$, we have
\[ \min \bigg( \frac{q}{\wt q}, \frac{r}{\wt r} \bigg) \leq \frac{3-s}{1-s} = p, \]
so that we can only find pairs $(q,r)$ and $(\wt q, \wt r)$ that satisfy $q = p \wt q$ and $r = p \wt r$ instead of $q > p \wt q$ and $r \geq p \wt r$. Such pairs do exist. One can take, for example, $(q,r)$ and $(\wt q, \wt r)$ as in \eqref{qr}.

In the case when $1 < p \leq 3$ and $s = s_{\text{crit}} = 0$, there does not exist any pair $(\wt q, \wt r)$ that satisfies $1 < \wt q \leq 2$, $1 \leq \wt r \leq 2$, and the scaling condition \eqref{str_cond} (with $d = 2$) simultaneously. In this case, the inhomogeneous Strichartz estimates (Lemma \ref{LEM:str}) no longer applies, so that an alternative approach is needed to deal with this case.
\end{remark}

\subsection{Local well-posedness of subcritical vNLW}

In this subsection, we prove the following theorem for the local well-posedness result of vNLW \eqref{vNLWcrit}.

\begin{theorem}\label{THM:LWP3}
Let $p > 1$ and let $s_{\text{crit}}$ be as in \eqref{scrit}. Then, $\eqref{vNLWcrit}$ is locally well-posed in $\H^s(\T^2)$ for 
\begin{center}
\textup{(i)} $p \neq 3$: $s \geq s_{\text{crit}}$ \qquad or \qquad \textup{(ii)} $p = 3$: $s > s_{\text{crit}}$.
\end{center}
More precisely, given any $(u_0, u_1) \in \H^s(\T^2)$, there exists $0 < T = T(u_0, u_1) \leq 1$ and a unique solution $\vec u = (u, \dt u)$ to \eqref{vNLWcrit} in the class
\[ (u, \dt u) \in C([0,T]; \H^s (\T^2)) \quad \text{and} \quad u \in L^q([0,T]; L^r(\T^2)), \]
for some suitable $q,r \geq 2$.
\end{theorem}

\begin{proof}
For the proof, we only consider the case $s < 1$. We first consider the case $s > s_{\text{crit}}$. We write \eqref{vNLWcrit} in the Duhamel formulation:
\begin{equation}
u(t) = \G(u) := V(t)(u_0,u_1) - \int_0^t S(t-t') F(u)(t') dt',
\label{LL1}
\end{equation}

\noi
where $F(u) = |u|^{p-1}u$, $V(t)$ is as defined in \eqref{defV}, and $S(t)$ is as defined in \eqref{defS}. Let $\vec \G(u) = (\G(u), \dt\G(u))$ and $\vec u = (u, \dt u)$. 

Let $(q,r)$ and $(\wt q, \wt r)$ be as given in Lemma \ref{LEM:qr}, which guarantees that $q > p\wt q$ and $r \geq p\wt r$. Given $0 < T \leq 1$, we define the space $\mathcal{Y}(T)$ as
\[ \mathcal{Y}^s(T) = \mathcal{Y}_1^s(T) \times \mathcal{Y}_2^s(T), \]
where
\begin{align*}
&\mathcal{Y}_1^s(T) := C([0,T]; H^s (\T^2)) \cap L^{q}([0,T]; L^{r}(\T^2)), \\
&\mathcal{Y}_2^s(T) := C([0,T]; H^{s-1} (\T^2)).
\end{align*}
Our goal is to show that $\vec \G$ is a contraction on a ball in $\mathcal{Y}^s(T)$ for some $0 < T \leq 1$.

By \eqref{LL1}, Lemma \ref{LEM:str}, and H\"older's inequality, we have
\begin{equation}
\begin{split}
\| \vec \G (u) \|_{\mathcal{Y}^s(T)} &\les \| (u_0, u_1) \|_{\H^s} + \big\| |u|^p \big\|_{L_T^{\wt q} L_x^{\wt r}} \\
&\les \| (u_0, u_1) \|_{\H^s} + T^\ta \| u \|_{L_T^q L_x^r}^p \\
&\les \| (u_0, u_1) \|_{\H^s} + T^\ta \| \vec u \|_{\mathcal{Y}^s(T)}^p
\end{split}
\label{LWPeq1}
\end{equation}

\noi
for some $\ta > 0$. 

For the difference estimate, we use the idea from Oh-Okamoto-Pocovnicu \cite{OOP}. Noticing that $F'(u) = p|u|^{p-1}$, we use \eqref{LL1}, Lemma \ref{LEM:str}, the fundamental theorem of calculus, Minkowski's integral inequality, and H\"older's inequality to obtain
\begin{align*}
\| \vec \G (u) - \vec \G (v) \|_{\mathcal{Y}^s(T)} &\les \| F(u) - F(v) \|_{L_T^{\wt q} L_x^{\wt r}} \\
&= \bigg\| \int_0^1 F'(v+\tau (u-v)) (u-v) d\tau \bigg\|_{L_T^{\wt q} L_x^{\wt r}} \\
&\les \int_0^1 \| v + \tau(u-v) \|_{L_T^{p\wt q} L_x^{p\wt r}}^{p-1} \| u-v \|_{L_T^{p\wt q} L_x^{p\wt r}} d\tau \\
&\les T^\ta \Big( \| u \|_{L^q_T L_x^r}^{p-1} + \| v \|_{L^q_T L_x^r}^{p-1} \Big) \| u-v \|_{L^q_T L_x^r} \\
&\les T^\ta \Big( \| \vec u \|_{\mathcal{Y}^s(T)}^{p-1} + \| \vec v \|_{\mathcal{Y}^s(T)}^{p-1} \Big) \| \vec u - \vec v \|_{\mathcal{Y}^s(T)}.
\end{align*}

\noi
for some $\ta > 0$.

Thus, by choosing $T = T ( \| (u_0,u_1) \|_{\H^s} ) > 0$ small enough, we obtain that $\vec \G$ is a contraction on the ball $B_R \subset \mathcal{Y}^s(T)$ of radius $R \sim 1 + \| (u_0,u_1) \|_{\H^s}$.

In the case when $p > 3$ and $s = s_{\text{crit}} = 1 - \frac{2}{p-1} > 0$, we can only find pairs $(q,r)$ and $(\wt q, \wt r)$ that satisfy $q = p \wt q$ and $r = p \wt r$ (see Remark \ref{REM:qr}). In this case, we modify the argument as follows. By \eqref{LL1}, Lemma \ref{LEM:str}, and H\"older's inequality, we obtain
\begin{align*}
\| \G (u) \|_{L_T^q L_x^r} &\les \| V(t) (u_0, u_1) \|_{L_T^q L_x^r} + \big\| |u|^p \big\|_{L_T^{\wt q} L_x^{\wt r}} \\
&\les \| V(t) (u_0, u_1) \|_{L_T^q L_x^r} + \| u \|_{L_T^{q} L_x^{r}}^p 
\end{align*}
for some $\ta > 0$ and sufficiently small $\eps > 0$. A difference estimate on $\G(u) - \G(v)$ also holds by a similar computation. By the dominated convergence theorem, we have $\| u \|_{L_T^{q} L_x^{r}}^p \to 0$ as $T \to 0$. Thus, we can choose $T = T(u_0, u_1) > 0$ sufficiently small such that $\| V(t) (u_0, u_1) \|_{L_T^q L_x^r} \leq \frac 12 \eta \ll 1$, so that we can show that $\G$ is a contraction on the ball of radius $\eta$ in $L_T^q L_x^r$. Moreover, \eqref{LWPeq1} gives
\[ \| \vec u \|_{C_T \H_x^s} = \| \vec \G(u) \|_{C_T \H_x^s} \les \| (u_0, u_1) \|_{\H^s} + \| u \|_{L_T^{q} L_x^{r}}^p < \infty, \]
so that $\vec u = (u, \dt u) \in C_T \H_x^s$.

Lastly, we consider the case when $1 < p < 3$ and $s = s_{\text{crit}} = 0$. Note that $s = 0$ along with the $L_T^3 L_x^3$ norm satisfies the scaling condition \eqref{hom_cond} in Lemma \ref{LEM:hom_str}. By \eqref{LL1}, Minkowski's integral inequality, Lemma \ref{LEM:hom_str}, Sobolev's inequality, and H\"older's inequality, we obtain
\begin{align*}
\| \G (u) \|_{L_T^3 L_x^3} &\les \| V(t) (u_0, u_1) \|_{L_T^3 L_x^3} + \int_0^T \big\| \ind_{[0,t]} (t') S(t-t') \big(|u|^{p-1}u \big) (t') \big\|_{L_T^3 L_x^3} dt' \\
&\les \| (u_0, u_1) \|_{\H^0} + \int_0^T \big\| \big( |u|^{p-1} u \big) (t') \big\|_{H_x^{-1}} dt' \\
&\les \| (u_0, u_1) \|_{\H^0} + \big\| |u|^p \big\|_{L_T^1 L_x^{1+}} \\
&\les \| (u_0, u_1) \|_{\H^0} + T^\ta \| u \|_{L_T^3 L_x^3}
\end{align*}
for some $\ta > 0$. Also, by \eqref{defV} and \eqref{defS}, we easily obtain
\[ \| \vec \G (u) \|_{C_T \H^0_x} \les \| (u_0, u_1) \|_{\H^0} + T^\ta \| u \|_{L_T^3 L_x^3}. \]
Similar difference estimates also hold, so that we can conclude using the standard contraction argument. This finishes the proof.
\end{proof}

We finish this appendix by stating several remarks.

\begin{remark}\label{REM:LWP3}\rm
(i) At this point, we do not know how to prove local well-posedness for the cubic vNLW (with $p = 3$) in $L^2(\T^2)$, i.e. with $s = s_{\text{crit}} = 0$. It would be of interest to investigate if spaces of functions of bounded $p$-variation (i.e. $U^p$- and $V^p$-spaces) such as those in \cite{BOP2, OOP} can be applied to handle the cubic case.

\smallskip \noi
(ii) A slight modification of the proof of Theorem \ref{THM:LWP3} yields local well-posedness of SvNLW \eqref{SvNLW} in $\H^s(\T^2)$ for all $s \geq s_{\text{crit}}$ (with a strict inequality when $p = 3$), which improves the local well-posedness result for SvNLW \eqref{SvNLW} in Theorem \ref{THM:LWP1}.

\smallskip \noi
(iii) One can compare the local well-posedness result for vNLW \eqref{vNLWcrit} in Theorem \ref{THM:LWP3} with the local well-posedness result for the usual NLW (see Remark 1.4 in \cite{GKO}):
\[ \dt^2 u - \Dl u \pm |u|^{p-1}u = 0. \]
Note that vNLW enjoys a better local well-posedness result than does the usual NLW, thanks to the parabolic smoothing effect.

\smallskip \noi
(iv) Note that the global well-posedness result of SvNLW \eqref{SvNLW} in Theorem \ref{THM:GWP1} easily gives global well-posedness of vNLW \eqref{vNLWcrit} in the class $\H^s(\T^2)$ for $s \geq \max (0, 1 - \frac{1}{p+\dl} - \frac 1p)$, where $\dl > 0$ is arbitrary. However, at this point, we do not know how to prove global well-posedness of vNLW \eqref{vNLWcrit} in $\H^s(\T^2)$ for $s_{\text{crit}} \leq s < \max (0, 1 - \frac{1}{p+\dl} - \frac 1p)$. The main difficulty for this range of $s$ is showing $\vec v(t) \in \H^1(\T^2)$ for all small enough $t > 0$, which is needed to guarantee the finiteness of the energy $E(\vec v)$ defined in \eqref{defE}.
\end{remark}

\begin{ackno}\rm
The author would like to thank his advisor, Tadahiro Oh, for suggesting this problem and for his support and advice throughout the whole work. The author is also grateful to Gyu Eun Lee and Guangqu Zheng for helpful suggestions and discussions. In addition, the author thanks the anonymous referees for the helpful comments. R.L. was supported by the European Research Council (grant no. 864138 ``SingStochDispDyn'').
\end{ackno}


\begin{thebibliography}{99}

\bibitem{BCD}
H.~Bahouri, J.Y.~Chemin, R. Danchin,
{\it Fourier Analysis and Nonlinear Partial Differential Equations},
Grundlehren der mathematischen Wissenschaften. Springer Berlin Heidelberg, 2011.

\bibitem{BO}
\'A.~B\'enyi, T.~Oh,
{\it The Sobolev inequality on the torus revisited},
Publ. Math. Debrecen, 83 (2013), no. 3, 359--374.


\bibitem{BOP2}
\'A.~B\'enyi, T.~Oh, O.~Pocovnicu,
{\it On the probabilistic Cauchy theory of the cubic nonlinear Schr\"odinger equation on $\R^3$, $d \geq 3$}, Trans. Amer. Math. Soc. Ser. B, 2 (2015), 1--50.



\bibitem{Bony}
J.-M.~Bony, 
{\it Calcul symbolique et propagation des singularit\'es pour les \'equations aux d\'eriv\'ees partielles non lin\'eaires},
Ann. Sci. \'Ecole Norm. Sup., 14 (1981), no. 2, 209--246.


\bibitem{BO96}
J.~Bourgain, 
{\it Invariant measures for the 2D-defocusing nonlinear Schr\"odinger equation}, 
Comm. Math. Phys., 176 (1996), no. 2, 421--445. 

\bibitem{BT08}
N.~Burq, N.~Tzvetkov,
{\it Random data Cauchy theory for supercritical wave equations. I. Local theory}, 
Invent. Math., 173 (2008), no. 3, 449--475.


\bibitem{BT14}
N.~Burq, N.~Tzvetkov,
{\it Probabilistic well-posedness for the cubic wave equation},
J. Eur. Math. Soc., 16 (2014), no. 1, 1--30.

\bibitem{CK}
R.~Carles, T.~Kappeler, 
{\it Norm-inflation with infinite loss of regularity for periodic NLS equations in negative Sobolev spaces}, Bull. Soc. Math. France, 145 (2017), no. 4, 623--642.

\bibitem{CP}
A.~Choffrut, O.~Pocovnicu, 
{\it Ill-posendess of the cubic nonlinear half-wave equation and other fractional NLS on the real line}, Int. Math. Res. Not. IMRN, (2018), no. 3, 699--738.

\bibitem{CCT}
M.~Christ, J.~Colliander, T.~Tao, 
{\it Ill-posedness for nonlinear Schr\"odinger and wave equations}, arXiv:0311048 [math.AP] (2003).


\bibitem{DPD}
G.~Da Prato, A.~Debussche, 
{\it Strong solutions to the stochastic quantization equations}, Ann. Probab., 31 (2003), no. 4, 1900--1916.

\bibitem{dO}
P.~de Roubin, M.~Okamoto, 
{\it Norm inflation for the viscous wave equations in negative Sobolev spaces}, in preparation.

\bibitem{FO}
J.~Forlano, M.~Okamoto, 
{\it A remark on norm inflation for nonlinear wave equations}, Dyn. Partial Differ. Equ., 17 (2020), no. 4, 361--381.


\bibitem{Gat}
A.E.~Gatto, 
{\it Product rule and chain rule estimates for fractional derivatives on spaces that satisfy the doubling condition}, J. Funct. Anal., 188 (2002), no. 1, 27--37.

\bibitem{GV}
J.~Ginibre, G.~Velo,
{\it Generalized Strichartz inequalities for the wave equation}, J. Funct. Anal., 133 (1995), 50--68.

\bibitem{GIP}
M.~Gubinelli, P.~Imkeller, N.~Perkowski,
{\it Paracontrolled distributions and singular PDEs},
Forum Math. Pi, 3 (2015), e6, 75 pp.

\bibitem{GKO}
M.~Gubinelli, H.~Koch, T.~Oh,
{\it  Renormalization of the two-dimensional stochastic nonlinear wave equations},
 Trans. Amer. Math. Soc.,
 370 (2018), no 10, 7335--7359.

\bibitem{GKO2}
M.~Gubinelli, H.~Koch, T.~Oh,
{\it Paracontrolled approach to the three-dimensional stochastic nonlinear wave equation with quadratic nonlinearity}, 
to appear in J. Eur. Math. Soc.


\bibitem{KT}
M.~Keel, T.~Tao,
{\it Endpoint Strichartz estimates}, Amer. J. Math., 120 (1998), no. 5, 955--980.

\bibitem{Kish19}
N.~Kishimoto, 
{\it A remark on norm inflation for nonlinear Schr\"odinger equations}, Commun. Pure Appl. Anal., 18 (2019), no. 3, 1375--1402.

\bibitem{KC1}
J.~Kuan, S.~\v{C}ani\'c, 
{\it Deterministic ill-posedness and probabilistic well-posedness of the viscous nonlinear wave equation describing fluid-structure interaction}, Trans. Amer. Math. Soc., 374 (2021), 5925--5994

\bibitem{KC2}
J.~Kuan, S.~\v{C}ani\'c, 
{\it A stochastically perturbed fluid-structure interaction problem modeled by a stochastic viscous wave equation}, 
J. Differential Equations 310 (2022), 45--98.

\bibitem{CKO}
J.~Kuan, T.~Oh,
{\it Probabilistic global well-posedness for a viscous nonlinear wave equation modeling fluid-structure interaction},
Appl. Anal. 101 (2022), no. 12, 4349--4373.

\bibitem{Lat}
M.~Latocca, 
{\it Almost sure existence of global solutions for supercritical semilinear wave equations}, Journal of Differential Equations, 273 (2021): 83--121.

\bibitem{LS}
H.~Lindblad, C.~Sogge, 
{\it On existence and scattering with minimal regularity for semilinear wave equations}, J. Funct. Anal., 130 (1995), 357--426.

\bibitem{LO}
R.~Liu, T.~Oh,
{\it On the two-dimensional singular stochastic viscous nonlinear wave equations},
C. R. Math. Acad. Sci. Paris 360 (2022), 1227--1248.


\bibitem{LM2}
J.~L\"uhrmann, D.~Mendelson,
{\it On the almost sure global well-posedness of energy sub-critical nonlinear wave equations on $\R^3$}, New York Journal of Mathematics, 22 (2016), 209--227.

\bibitem{McKean}
H.P.~McKean, 
{\it Statistical mechanics of nonlinear wave equations. IV. Cubic Schr\"odinger}, 
 Comm. Math. Phys., 168 (1995), no. 3, 479--491. 
 {\it Erratum: Statistical mechanics of nonlinear wave equations. IV. Cubic Schr\"odinger}, Comm. Math. Phys., 173 (1995), no. 3, 675.

\bibitem{MPTW}
R.~Mosincat, O.~Pocovnicu, L.~Tolomeo, Y.~Wang,
{\it Well-posedness theory for the three-dimensional stochastic nonlinear beam equation with additive space-time white noise forcing}, preprint.

\bibitem{MW}
J.-C.~Mourrat, H.~Weber,
{\it Global well-posedness of the dynamic $\Phi^4$ model in the plane},
Ann. Probab., 45 (2017), no. 4, 2398--2476.


\bibitem{Oh}
T.~Oh, 
{\it A remark on norm inflation with general initial data for the cubic nonlinear Schr\"odinger equations in negative Sobolev spaces}, Funkcial. Ekvac., 60 (2017), 259--277.

\bibitem{OOP}
T.~Oh, M.~Okamoto, O.~Pocovnicu,
{\it On the probabilistic well-posedness of the nonlinear Schr\"odinger equations with non-algebraic nonlinearities}, Discrete Contin. Dyn. Syst. A., 39 (2019), no. 6, 3479--3520.

\bibitem{OOTz}
T.~Oh, M.~Okamoto, N.~Tzvetkov,
{\it  Uniqueness and non-uniqueness of the Gaussian free field evolution under the two-dimensional Wick ordered cubic wave equation}, preprint.

\bibitem{OP16}
T.~Oh, O.~Pocovnicu,
{\it  Probabilistic global well-posedness of the energy-critical defocusing quintic nonlinear wave equation on $\R^3$}, J. Math. Pures Appl., 105 (2016), 342--366. 

\bibitem{OP17}
T.~Oh, O.~Pocovnicu,
{\it A remark on almost sure global well-posedness of the energy-critical defocusing nonlinear wave equations in the periodic setting}, 
Tohoku Math. J., 69 (2017), no. 3, 455--481.


\bibitem{OW}
T.~Oh, Y.~Wang, 
{\it On the ill-posedness of the cubic nonlinear Schr\"dinger equation on the circle}, An. \c{S}tiin\c{t}. Univ. Al. I. Cuza Ia\c{s}i. Mat. (N.S.), 64 (2018), no. 1, 53--84.

\bibitem{Oka}
M.~Okamoto, 
{\it Norm inflation for the generalized Boussinesq and Kawahara equations}, Nonlinear Anal., 157 (2017), 44--61.


\bibitem{Poc17}
O.~Pocovnicu,
{\it Probabilistic global well-posedness of the energy-critical defocusing cubic non-
linear wave equations on $\R^4$}, J. Eur. Math. Soc. (JEMS), 19 (2017), 2321--2375.

\bibitem{SX}
C.~Sun, B.~Xia,
{\it Probabilistic well-posedness for supercritical wave equations with periodic boundary condition on dimension three}, Illinois Journal of Mathematics, 60 (2016), no. 2, 481--503.

\bibitem{TAO}
T.~Tao,
{\it Nonlinear dispersive equations: Local and global analysis}, CBMS Regional Conference Series in Mathematics, vol. 106, Amer. Math. Soc., Providence, RI, 2006. xvi+373 pp.

\bibitem{Tz}
N.~Tzvetkov, 
{\it Random data wave equations}, in F. Flandoli, M. Gubinelli, and M. Hairer, editors, Singular random dynamics, volume 2253 of Lecture Notes in Mathematics, pages 221--313. Springer, 2019.



\end{thebibliography}
\end{document}